\documentclass[11pt]{article} 

\usepackage{amsmath}
\usepackage{amssymb}
\usepackage{graphicx}
\usepackage{color}
\def\r{\rightarrow}

\newcommand{\fdem}{\hspace*{\fill}~$\Box$\par\endtrivlist\unskip}

\newcommand{\E}{\mathbb{E}}     
\renewcommand{\P}{\mathbb{P}}     
\renewcommand{\L}{\mathbb{L}}

\newcommand{\N}{\mathbb{N}}     

\newcommand{\R}{\mathbb{R}}     
     
\newcommand{\C}{\mathbb{C}}

\newcommand{\X}{\mathbb{X}}

\renewcommand{\dim}{\mathop{\rm dim}}

\renewcommand{\r}{\mathop{\rightarrow}}

\newcommand{\cB}{\mbox{$\cal B$}}
\newcommand{\cC}{\mbox{$\cal C$}}

\newcommand{\cF}{\mbox{$\cal F$}}

\newcommand{\cN}{\mbox{$\cal N$}}

\newcommand{\cX}{\mbox{$\cal X$}}

\newcommand{\Pc}{\widehat{P}}
\newcommand{\Qc}{\widehat{Q}}

\newenvironment{proof}[1]{\textit{Proof#1.\,}}{\fdem}

\newtheorem{atheo}{Theorem}[section]

\newtheorem{alem}[atheo]{Lemma}
\newtheorem{arem}[atheo]{Remark}
\newtheorem{acor}[atheo]{Corollary}
\newtheorem{apro}[atheo]{Proposition}

\title{State-discretization of $V$-geometrically ergodic Markov chains and convergence to the stationary distribution}
\author{Loic HERV\'E and James LEDOUX \footnote{Univ Rennes, INSA Rennes, CNRS, IRMAR-UMR 6625,  F-35000 Rennes, France. 
 Loic.Herve@insa-rennes.fr, James.Ledoux@insa-rennes.fr}
}

\begin{document}

\maketitle
\vspace*{-4mm}
\begin{abstract}
Let $(X_n)_{n \in\N}$ be a $V$-geometrically ergodic Markov chain on a measurable space $\X$ with invariant probability distribution $\pi$. In this paper, we propose a discretization scheme providing a computable sequence $(\widehat\pi_k)_{k\ge 1}$ of probability measures which approximates $\pi$ as $k$ growths to infinity. The probability measure $\widehat\pi_k$ is computed from the invariant probability distribution of a finite Markov chain. The convergence rate in total variation of  $(\widehat\pi_k)_{k\ge 1}$ to $\pi$ is given. As a result, the specific case of first order autoregressive processes with linear and non-linear errors is studied. Finally, illustrations of the procedure for such autoregressive processes are provided, in particular when no explicit formula for $\pi$ is known. 
\end{abstract}
\vspace*{-5mm}
\begin{center}


\end{center}


\section{Introduction} 
Let $(\X,d)$ denote a metric space equipped with its Borel $\sigma$-algebra $\cX$. Let $(X_n)_{n\in\N}$ be a Markov chain with state space $(\X,\cX)$ and transition kernel $P$ of the form
\begin{equation} \label{form-P}
\forall x\in\X,\quad P(x,dy) = p(x,y)\, d\mu(y),
\end{equation}
where $p : \X^2\r[0,+\infty)$ is a measurable function and $\mu$ is a positive $\sigma$-additive  
 measure on $(\X,\cX)$. Typically $\X$ is $\R^d$ and $\mu$ is the Lebesgue measure on $\R^d$. Moreover let $v : [0,+\infty)\r[1,+\infty)$ denote an unbounded increasing continuous function such that $v(0)=1$, and let $V : \X\r[1,+\infty)$ be defined by 
\begin{equation} \label{v-V}
\forall x\in\X,\quad V(x) := v\big(d(x,x_0)\big),
\end{equation}
where $x_0\in\X$ is fixed. We assume that $P$ admits an invariant probability measure $\pi$ on $(\X,\cX)$. Since $P$ is of the form (\ref{form-P}), $\pi$ is absolutely continuous with respect to $\mu$, that is $d\pi(y) = \mathfrak{p}(y)\, d\mu(y)$ for some probability density function (pdf) $\mathfrak{p}$. Throughout the paper we assume that 
$$\pi(V) := \int_\X V(y)\, \mathfrak{p}(y)\, d\mu(y)<\infty$$ 
and that $P$ is $V$-geometrically ergodic, that is (e.g. see \cite{MeyTwe93}): there exist $\rho\in(0,1)$ and a positive constant $C\equiv C(\rho)$ such that the following inequality holds for every measurable complex-valued function $f$ on  $\X$ satisfying $|f| \leq V$: 
\begin{equation} \label{V-geo}
\forall n\geq 0,\quad \sup_{x\in\X} \frac{|(P^nf)(x) - \pi(f)|}{V(x)} \leq C\, \rho^n.
\end{equation}
Mention that, for most of the classical $V$-geometrically ergodic Markov chains, the function $V$ is of the form (\ref{v-V}).

Even for simple models as first-order autoregressive models, the explicit computation of the stationary pdf $\mathfrak{p}$ is a difficult issue, and it is only possible for some specific examples. 
In this work, under suitable assumptions on the kernel $p(x,y)$, we propose a discretization procedure  providing a computable sequence $(\widehat\pi_k)_ {k \ge 1}$ of probability measures on $\X$ which approximates the stationary distribution $\pi$ of $P$ in total variation distance. Roughly speaking the probability measure $\widehat\pi_k$ on $\X$ is defined as follows. For every integer $k$, an explicit finite stochastic matrix 
$B_k$ 
is derived from the Markov kernel $P$ by discretization of the kernel $p(x,y)$. Then $\widehat\pi_k$ is defined as a natural extension of the left $B_k$-invariant probability vector. Then the above mentioned convergence of $(\widehat\pi_k)_ {k \ge 1}$ to $ \pi$ in total variation distance is derived in Theorem~\ref{theo-main} from the results of \cite{HerLed14}. Moreover the absolutely continuous part $\mathfrak{p}_k$ of $\widehat\pi_k$ w.r.t.~$\mu$ can be explicitly computed, and the sequence $(\mathfrak{p}_k)_ {k \ge 1}$ is proved to converge to $\mathfrak{p}$ in the usual Lebesgue space $\L^1(\X,\cX,\mu)$ (see Corollary~\ref{cor-main}). 
Applications to the first order (linear) autoregressive models AR(1) and to AR(1) processes with ARCH$(1)$ errors are addressed in Sections~\ref{sec-AR}. The computational issues to get $\mathfrak{p}_k$ are discussed in Section~\ref{ap-algo-ar}. Numerical illustrations are presented in Section~\ref{sec-examp}.

The authors in \cite{Hai98,AndHra00,AndHra07} developed another method to approximate the stationary pdf $\mathfrak{p}$ of linear processes. Their approach consists in approximating 
the stationary pdf $\mathfrak{p}$ of an AR(1) process (i.e.~$X_n = \varrho\, X_{n-1} + \vartheta_n$, see Subsection~\ref{sub_AR1} for details) 
by the sequence $(h_n)_{n \in\N}$ of functions recursively defined by 
\begin{equation} \label{eq-hn}
h_0:=\nu\qquad \text{and} \qquad \forall n\geq1,\ \ h_n(x) := \int_\R \nu(x-\varrho u)\, h_{n-1}(u)\, du 
\end{equation}
where $\nu$ denotes the innovation pdf  (i.e.~the law of $\vartheta_1$).
Hainman in \cite{Hai98} proved that $(h_n)_{n \in\N}$ uniformly converges to $\mathfrak{p}$ with geometric rate under strong assumptions on the support of $\nu$. The authors in \cite{AndHra00} proved that $(h_n)_{n \in\N}$ converges point-wise to $\mathfrak{p}$ under some mild assumptions on the Fourier transform of $\nu$, and they established the uniform convergence with geometric rate in the case when $\nu$ is the exponential pdf. In \cite{AndHra07} the uniform convergence of $(h_n)_{n \in\N}$  to $\mathfrak{p}$, with geometric rate, is extended to general causal linear processes under mild assumptions on the noise process. Closely linked to these works, we also mention the paper \cite{Log04} which studies the characteristic function of the stationary pdf $\mathfrak{p}$ for a threshold AR(1) model with noise process having Laplace distribution, as well as the paper \cite{AndRan05} which investigates $\mathfrak{p}$ for absolute autoregressive associated with noise process having Gaussian, Cauchy or Laplace distribution (from \cite{ChaTon86} this issue may be reduced to the computation of the stationary pdf of an auxiliary AR(1) process).

Due to \cite{AndHra07}, the approximation of $\mathfrak{p}$ by $(h_n)_{n \in\N}$ via Equation~(\ref{eq-hn}) is theoretically efficient for linear processes since the rate of convergence  is geometric. However, except when the noise process has a special usual law, the exact calculation of the integral in (\ref{eq-hn}) can not be carried out. Moreover, any numerical method recursively providing approximations of the integrals $h_1,\ldots,h_p$ for some $p\geq 1$ induces some cumulative  errors. For linear processes our method is thus an alternative way to approximate $\mathfrak{p}$: the rate of convergence in our work is not geometric (a priori), but for some $k\geq 1$ the approximation $\mathfrak{p}_k$ of $\mathfrak{p}$ as above described can be directly computed (without any recursive procedure).  Section~\ref{sec-examp} provides numerical evidence for robustness of the method. Moreover our approach applies to any $V$-geometrical Markov chain (not only to linear processes) admitting a probability kernel $P(x,dy)$ of the form (\ref{form-P}), provided that the kernel $p(\cdot,\cdot)$ has some suitable Lipschitz-regularity properties (see Assumption~(\ref{Lk-finite})). For instance our method applies to autoregressive process with ARCH$(1)$ errors (see Subsections~\ref{sec-arch} and \ref{subsec-num-arch}).

The invariant pdf $\mathfrak{p}$ satisfies the functional equation $T \mathfrak{p}= \mathfrak{p}$, where $T$ is the linear operator defined by $(Tf)(\cdot) = \int_\X p(y,\cdot)\, f(y)\, d\mu(y)$. However this operator $T$ is not used in this work. Indeed that is not $T$, but $P$, which is approximated by a sequence of finite-rank operators $(\Pc_k)_{k\ge 1}$. The reason for this is that $P$ has good spectral properties on the usual weighted-supremum Banach space $\cB_1$ associated with $V$ due to the $V$-geometrical ergodicity assumption. Also note that the classical theory of perturbed operators does not apply here because the sequence $(\Pc_k)_{k\ge 1}$  does not converge to $P$ for the usual operator norm on $\cB_1$ (in particular $P$ is not a compact operator on $\cB_1$). To get around this difficulty, we use the results of \cite{HerLed14} based on the Keller-Liverani perturbation theorem \cite{KelLiv99}: this method requires an auxiliary weaker operator norm on $\cB_1$ (see Lemma~\ref{lem-appro}), as well as uniform (in $k$) drift inequalities for $\Pc_k$ (see Lemma~\ref{lem-drift}). In the context of perturbed $V$-geometrically ergodic Markov chains, the interest of using an auxiliary norm appears in \cite{ShaStu00} (see \cite{Kel82} for similar issues in ergodic theory). For recent works related to this weak perturbation method in Markovian models, see \cite{FerHerLed13,RudSch18,Tru17} and the references therein.  

\section{Definition of the approximating probability measure $\widehat\pi_k$} \label{sect-disc}

Let $x_0\in\X$ be fixed and, for every integer $k\geq1$, let us consider any $\X_k\in\cX$ such that   
$$\big\{x\in\X\, :\, d(x,x_0) < k\big\}\, \subseteq\, \X_k\,  \subseteq\, \big\{x\in\X\, :\, d(x,x_0) \leq k\big\}.$$

Let us introduce the following finite partitions of the sequence of spaces $(\X_k)_{k\geq 1}$. 

\noindent{\bf Definition (A).} {\it Let $(\delta_k)_{k\geq 1}$  be a sequence of positive real numbers such that $\lim_k\delta_k=0$. For every integer $k\geq1$, we consider a finite family $\{\X_{j,k}\}_{j\in I_k}$ of disjoint measurable subsets of $\X_k$ such that 
\begin{equation} \label{cond-diam}
\X_k = \bigsqcup_{j\in I_k} \X_{j,k} \qquad \text{with}\ \ \forall j\in I_k,\quad \text{diam}( \X_{j,k}) \leq \delta_k.
\end{equation}
where $\text{diam}( \X_{j,k}) := \sup\big\{d(x,x')\, :\, (x,x')\in\X_{j,k}\big\}$. The positive real number $\delta_k$ must be thought of as  the mesh of the partition $\{\X_{j,k}\}_{j\in I_k}$.
}

Define   
$$\forall k\geq1,\ \forall (x,y)\in\X^2,\quad p_k(x,y) := 1_{\X_k}(y)\sum_{i\in I_k} 1_{\X_{i,k}}(x) \, \inf_{t\in \X_{i,k}} p(t,y),$$
Observe that $p_k \leq p$. 
Below $f : \X\r\C$ denotes any bounded measurable function on $\X$ where $\C$ denoted the set of complex numbers. We define the following non-negative kernel $\Qc_k$: 
\begin{eqnarray}
\forall x\in\X,\quad (\Qc_k f)(x) &:=& \int_\X f(y)\, p_k(x,y)\, d\mu(y) \nonumber \\
&=& \sum_{i\in I_k} \bigg(\int_{\X_k} f(y)\, \inf_{t\in \X_{i,k}} p(t,y)\, d\mu(y)\bigg)1_{\X_{i,k}}(x). \label{dec-Pkf}
\end{eqnarray}
Note that $\Qc_k f$ vanishes on $\X\setminus \X_k$. Let $\psi_k$ be the non-negative function on $\X$ defined by 
$$\psi_k := 1_{\X} - \Qc_k1_{\X}.$$
We have $\psi_k\equiv 1$ on $\X\setminus \X_k$, and  $0\leq \psi_k \leq 1_{\X}$ since $0\leq \Qc_k1_{\X} \leq P1_{\X}=1_{\X}$.   
Next define the following kernel: 
\begin{equation} \label{def-hatPk}
\forall x\in\X,\quad (\Pc_kf)(x) := (\Qc_k f)(x) + f(x_0)\, \psi_k(x).
\end{equation}
Then $\Pc_k$ is a Markov kernel on $(\X,\cX)$, i.e.~$\Pc_k$ is non-negative ($f\geq0 \Rightarrow \Pc_kf\geq0$) and $\Pc_k1_{\X} = 1_{\X}$. 

Moreover we deduce from (\ref{def-hatPk}) and (\ref{dec-Pkf}) that $\Pc_k(f) \in \cF_k$, where $\cF_k$ is the finite-dimensional space spanned by the system of functions  
$\big\{1_{\X_{i,k}},\ i\in I_k\big\}\cup\{\psi_k\}$. 
Observe that $1_{\X}\in \cF_k$ from $1_{\X} = \Qc_k1_{\X}+\psi_k$ and  (\ref{dec-Pkf}). Now define 
\vspace*{-2mm}
$$b_{k} := 1_{\X}-1_{\X_k} = 1_{\X\setminus\X_k}.$$
Then $b_k\in\cF_k$ since $1_{\X}\in\cF_k$ and $b_k = 1_{\X} - \sum_{i\in I_k} 1_{\X_{i,k}}$. 
Thus another basis of $\cF_k$ is given by 
\begin{equation} \label{Vk} 
\cC_k := \big\{1_{\X_{i,k}},\ i\in I_k\big\} \cup \{b_{k}\}.
\end{equation} 
Let $\{x_{i,k}\}_{i\in I_k}$ be such that $x_{i,k}\in\X_{i,k}$ and let  $\overline{x}_k\in\X\setminus\X_k$. Then we have for every $g\in\cF_k$: 
\begin{equation} \label{dec-Vk} 
g=\sum_{i\in I_k} g(x_{i,k})\, 1_{\X_{i,k}} \, + g(\overline{x}_k)\, b_k.
\end{equation} 

\vspace*{-6mm}
Now, from $\Pc_k(\cF_k) \subset \cF_k$ we can define the linear map $P_k : \cF_k\r\cF_k$ as the restriction of $\Pc_k$ to $\cF_k$. Let $N_k:=\dim\cF_k = \text{Card}\, (I_k)+1$, and let $B_k$ 
be the $N_k\times N_k-$matrix defined as the matrix of $P_k$ with respect to the basis $\cC_k$. Note that 
\begin{equation} \label{Bk-bk-zero} 
P_kb_k = \Pc_kb_k = \Qc_kb_k + b_k(x_0)\psi_k = 0,
\end{equation}  
and that for every $j\in I_k$ 
\begin{eqnarray*}
P_k 1_{\X_{j,k}} &=& \Pc_k 1_{\X_{j,k}} \\
&=& \sum_{i\in I_k} (\Pc_k 1_{\X_{j,k}})(x_{i,k})\, 1_{\X_{i,k}} \, + (\Pc_k 1_{\X_{j,k}} )(\overline{x}_k)\, b_k \qquad \qquad \text{(from (\ref{dec-Vk}))}\\ 
&=& \sum_{i\in I_k} \big[(\Qc_k 1_{\X_{j,k}})(x_{i,k}) + 1_{\X_{j,k}}(x_0)\, \psi_k(x_{i,k})\big]\, 1_{\X_{i,k}} + \big[(\Qc_k 1_{\X_{j,k}} )(\overline{x}_k) + 1_{\X_{j,k}}(x_0)\, \psi_k(\overline{x}_k)\big]\, b_k \\ 
&=& \sum_{i\in I_k} \big[(\Qc_k 1_{\X_{j,k}})(x_{i,k}) + 1_{\X_{j,k}}(x_0)\, \psi_k(x_{i,k})\big]\, 1_{\X_{i,k}} +  1_{\X_{j,k}}(x_0)\, b_k.
\end{eqnarray*}
The previous equalities show that $B_k$ is a non-negative matrix. Moreover Equality $P_k1_{\X} = 1_{\X}$ reads as matrix equality $B_k\cdot{\bf 1}_k = {\bf 1}_k$  
where ${\bf 1}_k$ is the coordinate vector of $1_{\X}$ in the basis $\cC_k$ and is given by ${\bf 1}_k = (1,\ldots,1)^{\top}$. The symbol $\cdot^{\top}$ stands for the transpose  operation.  
Thus $B_k$ is a stochastic matrix. Accordingly there exists a non-zero row-vector  $\pi_k\in[0,+\infty)^{N_k}$ such that  
\begin{equation} \label{lem-pik}
\pi_k\cdot B_k = \pi_k,\quad \text{and} \quad \pi_k\cdot{\bf 1}_k=1.
\end{equation}
Note that the last component of $\pi_k$ (i.e.~the component associated with $b_k$) is zero since the last column of $B_k$ is zero from (\ref{Bk-bk-zero}). We denote by  $\pi_{i,k}$ the component of $\pi_k$ associated with the element $1_{\X_{i,k}}$ of the basis $\cC_k$, so that the coordinate vector of $\pi_k$ in $\cC_k$ is $(\{\pi_{i,k}\}_{i\in I_k}\, ,\, 0)$. 
For every $k\geq 1$ we set 
\begin{equation} \label{def-hat-pik}
\widehat\pi_k(f) := \pi_k\cdot F_k
\end{equation}
where $F_k\equiv F_k(f)$ is the coordinate vector of $\Pc_kf$ in the basis $\cC_k$. 
\begin{apro} \label{def_approx}
$\widehat\pi_k$ defines a  $\Pc_k$-invariant probability measure on $(\X,\cX)$. Moreover we have    
\begin{equation} \label{dec-hat-pik}
\widehat\pi_k(dy) = \mathfrak{p}_k(y)\, d\mu(y) + \bigg(1-\int_\X\mathfrak{p}_k(y)\, d\mu(y)\bigg)\delta_{x_0},
\end{equation}
where $\delta_{x_0}$ is the Dirac distribution at $x_0$, and where $\mathfrak{p}_k$ is the non-negative function defined by 
\begin{equation} \label{def-density-pk} 
\forall y\in\X,\quad \mathfrak{p}_k(y) :=  1_{\X_k}(y)\sum_{i\in I_k} \pi_{i,k} \inf_{t\in \X_{i,k}} p(t,y),
\end{equation}
\end{apro}
Note that Formula (\ref{def-density-pk}) involves the infimum of the function $t \mapsto p(t,y)$ on each subset $\X_{i,k}$. This is a technical choice to ensure that Lemma~\ref{lem-drift} holds true for $\widehat P_k$. Specifically, this a simple choice to simplify the convergence analysis in Section \ref{sect-conv-gene} of  the approximation scheme.

\begin{proof}{} 
Recall that $b_k$ is defined by $b_k = 1_{\X} - \sum_{i\in I_k} 1_{\X_{i,k}}$. From $\psi_k := 1_{\X} - \Qc_k1_{\X}$ it follows that $\psi_k  = b_k + \sum_{i\in I_k} 1_{\X_{i,k}}- \Qc_k1_{\X}$. 
Define 
\begin{equation} \label{def-mik-f}
m_{i,k}(f) := \int_{\X_k} f(y)\, \inf_{t\in \X_{i,k}} p(t,y)\, d\mu(y)
\end{equation}
and observe that $\Qc_k f = \sum_{i\in I_k} m_{i,k}(f)\, 1_{\X_{i,k}}$. Then we deduce from (\ref{dec-Pkf}) and (\ref{def-hatPk}) that 
\begin{eqnarray*}
\Pc_kf := (\Qc_k f) + f(x_0)\, \psi_k &=& \sum_{i\in I_k} m_{i,k}(f)\, 1_{\X_{i,k}} + f(x_0)\big(b_k + \sum_{i\in I_k} 1_{\X_{i,k}}- \Qc_k1_{\X}\big) \\ 
&=& \sum_{i\in I_k} \big[m_{i,k}(f) + f(x_0) - f(x_0)\, m_{i,k}(1_{\X}) \big]\, 1_{\X_{i,k}} + f(x_0) b_k, 
\end{eqnarray*}
so that (\ref{def-hat-pik}) and $\sum_{i\in I_k} \pi_{i,k}=1$ give  
\begin{eqnarray}
\widehat\pi_k(f) &:=& \sum_{i\in I_k} \pi_{i,k}\, [ m_{i,k}(f) + f(x_0) - f(x_0)\, m_{i,k}(1_{\X}) \big] \nonumber \\
&=& \sum_{i\in I_k} \pi_{i,k}\,  m_{i,k}(f) + f(x_0)\bigg(1-\sum_{i\in I_k} \pi_{i,k}\, m_{i,k}(1_{\X})\bigg). \label{fle-hat-pi-k}
\end{eqnarray}
This proves Formula~(\ref{dec-hat-pik}). 
Now we prove that $\widehat\pi_k$ defines a  $\Pc_k$-invariant probability measure on $(\X,\cX)$. Note 
that 
$$\forall i\in I_k,\quad  m_{i,k}(1_{\X}) \leq \int_{\X} p(x_{i,k},y)\, d\mu(y) = (P1_{\X})(x_{i,k})=1,$$ 
thus  
$$\int_\X\mathfrak{p}_k(y)\, d\mu(y) = \sum_{i\in I_k} \pi_{i,k}\, m_{i,k}(1_{\X})\leq1.$$ 
It follows from this remark and from (\ref{fle-hat-pi-k}) that $\widehat\pi_k$ is a probability measure on $\X$. 
Finally $ B_k\cdot F_k$ is the coordinate vector of $\Pc_k^{\, 2}f$ in $\cC_k$ since $\Pc_kf\in \cF_k$ and $F_k$ is the coordinate vector of $\Pc_kf$ in $\cC_k$. 
Consequently  we deduce from (\ref{def-hat-pik}) and (\ref{lem-pik}) that 
$$\widehat\pi_k(\Pc_kf) := \pi_k\cdot B_k\cdot F_k = \pi_k\cdot F_k = \widehat\pi_k(f).$$
Thus $\widehat\pi_k$ is $\Pc_k$-invariant. 
\end{proof}

\section{Convergence of $(\widehat\pi_k)_ {k \ge 1}$ to $\pi$ in total variation distance} \label{sect-conv-gene}
The metric space $\X$ is equipped with a sequence of partitions satisfying Definition~{\bf (A)}. The Markov kernel $P$ on $\X$ is assumed to be of the form (\ref{form-P}). Let $\theta\in(0,1]$. For $k\geq 1$, $i\in I_k$, and $y\in\X_k$, we denote by $L_{i,k,\theta}(y)$ the following quantity in $[0,+\infty]$ 
\begin{equation} \label{L-i-k-theta}
L_{i,k,\theta}(y) := \sup\left\{\frac{|p(x,y) - p(x',y)|}{d(x,x')^\theta},\ (x,x')\in \X_{i,k}\times\X_{i,k},\ x\neq x'\right\}.
\end{equation}
Finally we assume that $P$ satisfies the following assumptions  
\begin{subequations}
\begin{gather} 
\exists\, \delta\in (0,1),\ \exists M\in(0,+\infty),\quad PV \leq \delta V + M 
\label{drift} \\[2mm]
\alpha_k := \sup_{u\in\X_k} \frac{P\big(u,\X\setminus\X_k\big)}{V(u)}  \longrightarrow 0 \quad \text{when}\ k\r+\infty \label{def-alpha} \\
\exists\, \theta\in(0,1],\ \forall k\geq 1,\ \ell_{k,\theta} := \max_{i\in I_k} 
\int_{\X_k} L_{i,k,\theta}(y) \, d\mu(y) < \infty, \quad \text{and} \ \lim_{k\r+\infty} \ell_{k,\theta} \, \delta_k^{\theta}  = 0. \label{Lk-finite} 
\end{gather}
\end{subequations}
Actually (\ref{drift}) is a drift type inequality (see \cite{MeyTwe93}) which comes from the $V$-geometric ergodicity assumption (\ref{V-geo}).  Technical conditions (\ref{def-alpha}) and (\ref{Lk-finite}) are used to control the weak convergence of $(\widehat P_k)_{k \ge 1}$ to $P$ (see Lemma~\ref{lem-appro}). In the first order autoregressive models of Section~\ref{sec-AR}, condition (\ref{def-alpha}) reduces to a polynomial moment condition on the noise (see (\ref{m}) for instance), and  Condition (\ref{Lk-finite}) reduces to the control of the derivative of the noise (see (\ref{lip}) for instance).
\begin{atheo} \label{theo-main}
Let $(\delta_k)_{k\ge 1}$ be a sequence of positive real numbers from Definition~{\bf (A)}. Assume that $P$ is a $V$-geometrically ergodic Markov kernel of the form (\ref{form-P}), and finally that Assumptions~(\ref{drift})-(\ref{Lk-finite}) hold. Then the probability measures $\widehat\pi_k$ on $\X$ given in (\ref{dec-hat-pik}) are such that $\|\pi-\widehat\pi_k\|_{TV}\r0$ when $k\r+\infty$, more precisely: 
\begin{equation}
\|\pi-\widehat\pi_k\|_{TV} = 
\text{O}\big(|\ln \tau_k|\, \tau_k\big) \quad \text{with} \quad \tau_k = 2\, \max\bigg(\frac{1}{v(k)}\ ,\, \alpha_k + 
\ell_{k,\theta}\,  \delta_k^\theta
\bigg). \label{estim-VT}
\end{equation}
\end{atheo}
Let $(\cB_0,\|\cdot\|_0)$ denote the Banach space of bounded measurable $\C$-valued functions on $\X$ equipped with the norm $\|f\|_0:=\sup_{x\in\X}|f(x)|$. Then (\ref{estim-VT}) means that 
$$\forall k\geq 1,\ \forall f\in\cB_0, \quad \big|\pi(f)-\widehat\pi_k(f)\big| \leq \gamma_k\, \|f\|_0$$
with $\gamma_k=\text{O}\big(|\ln \tau_k|\, \tau_k\big)$. Recall that $\pi(dy) = \mathfrak{p}(y)d\mu(y)$. Assume that $\mu(\{x_0\})=0$. Then, using (\ref{dec-hat-pik}), the previous inequalities applied to $f:=1_{\{x_0\}}$ imply
$$0\leq 1-\int_\X\mathfrak{p}_k(y)\, d\mu(y) \leq \gamma_k.$$ 
Hence 
$$\forall k\geq 1,\ \forall f\in\cB_0,\quad \big|\int_\X f(y)\, \mathfrak{p}(y)\, d\mu(y)  - \int_\X f(y)\,  \mathfrak{p}_k(y)\, d\mu(y) \big| \leq 2\, \gamma_k\, \|f\|_0,$$
from which we deduce the following corollary.
\begin{acor} \label{cor-main} 
Assume that the assumptions of Theorem~\ref{theo-main} hold and that $\mu(\{x_0\})=0$. Then the sequence $(\mathfrak{p}_k)_ {k \ge 1}$ given in (\ref{def-density-pk}) converges to $\mathfrak{p}$ in the usual Lebesgue space $\L^1(\X,\cX,\mu)$, more precisely  
\begin{equation} \label{estim-L1} 
\int_\X\big|\mathfrak{p}(y) - \mathfrak{p}_k(y)\big|\, d\mu(y) = \text{O}\big(|\ln \tau_k|\, \tau_k\big). 
\end{equation}
\end{acor}
\begin{proof}{ of Theorem~\ref{theo-main}} 
We apply \cite[Prop.~2.1(b)]{HerLed14} based on the Keller-Liverani perturbation theorem \cite{KelLiv99}. Define $(\cB_1,\|\cdot\|_1)$ as the weighted-supremum Banach space 
$$\cB_1 := \big\{ \ f : \X\r\C, \text{ measurable }: \|f\|_1  := \sup_{x\in\X} |f(x)| V(x)^{-1} < \infty\ \big\}.$$ 
Note that Inequality~(\ref{V-geo}) writes as follows
$$\forall n\geq 0,\ \forall f\in\cB_1,\quad  \|P^nf - \pi(f)\, 1_\X\|_1 \leq  C\, \rho^{n}\, \|f\|_1.$$ 
Since $p_k(x,y) \leq p(x,y)$, $\Pc_k$ continuously acts on both $\cB_0$ and $\cB_1$. In fact $\Pc_k$ is finite-rank, more precisely 
$$\Pc_k(\cB_1) \subset \cF_k$$
with $\cF_k$ given in Section~\ref{sect-disc} (see (\ref{Vk})). Note that $\widehat\pi_k$ clearly defines a non-negative bounded linear form on $\cB_1$. Then, according to \cite[Prop.~2.1(b)]{HerLed14}, Property~(\ref{estim-VT}) follows from the next Lemmas~\ref{lem-drift} and \ref{lem-appro}. 
\end{proof}
\begin{alem} \label{lem-drift}
We have 
$$\forall k\geq 1,\quad \Pc_kV \leq \delta V + L \qquad \text{with $\ L:=M+1$ and $M$ given in (\ref{drift})}.$$
\end{alem}
\begin{proof}{}
If $x\in \X\setminus\X_k$, then $(\Pc_k V)(x)= V(x_0)\, \psi_k(x)\leq1$. If $x\in\X_k$, then we obtain (see (\ref{def-mik-f})): 
\begin{eqnarray*}
(\Pc_k V)(x) &=& \sum_{i\in I_k} m_{i,k}(V)\, 1_{\X_{i,k}}(x) + V(x_0)\, \psi_k(x)\\
&\leq& \sum_{i\in I_k} \bigg(\int_{\X} V(y)\, p(x,y)\, d\mu(y)\bigg)1_{\X_{i,k}}(x) + 1 \\
&\leq&  (PV)(x) + 1. 
\end{eqnarray*}
The desired inequality follows from (\ref{drift}). 
\end{proof}

\begin{alem} \label{lem-appro}
For every $k\geq 1$ we have: $\displaystyle 
\sup_{f\in{\cal B}_0,\, \|f\|_0\leq 1} \|\Pc_kf- Pf\|_1 \ \leq\, \tau_k$.
\end{alem}
\begin{proof}{} 
Let $f\in{\cal B}_0$, $\|f\|_0\leq 1$. If $x\in \X\setminus\X_k$, it follows from $(\Pc_k f)(x)= f(x_0)\, \psi_k(x)$ that 
\begin{equation} \label{ineg-1}
\frac{\big|(\Pc_kf)(x)  - (Pf)(x)\big|}{V(x)} \leq \frac{\psi_k(x)+ (P|f|)(x)}{V(x)} \leq \frac{2}{V(x)} \leq \frac{2}{v(k)}.
\end{equation}
Next assume that $x\in\X_k$. Then we obtain from the definition of $\Qc_k$ that 
\begin{eqnarray*}
\frac{\big|(\Qc_kf)(x)  - (Pf)(x)\big|}{V(x)} &\leq& \int_\X \frac{\big|p_k(x,y) - p(x,y) \big|}{V(x)}\, d\mu(y)  \nonumber \\
&\leq&  \underbrace{\int_{\X\setminus\X_k} \frac{\big|p_k(x,y) - p(x,y) \big|}{V(x)}\, d\mu(y)}_{:= \alpha_k(x)} + \underbrace{\int_{\X_k} \frac{\big|p_k(x,y) - p(x,y) \big|}{V(x)}\, d\mu(y)}_{:= \beta_k(x)} 
\end{eqnarray*}
Since $p_k(x,y)=0$ when $y\in\X\setminus\X_k$, we obtain that 
$$\alpha_k(x) = \int_{\X\setminus\X_k} \frac{p(x,y)}{V(x)}\, d\mu(y) = \frac{P\big(x,\X\setminus\X_k\big)}{V(x)} \leq \alpha_k$$
from the definition (\ref{def-alpha}) of $\alpha_k$. Now, since $V\geq 1$, it follows from Conditions~(\ref{cond-diam}) and (\ref{Lk-finite})  that  

\begin{eqnarray*}
\beta_k(x) &=& \int_{\X_k} \bigg|\sum_{i\in I_k} 1_{\X_{i,k}}(x) \inf_{t\in \X_{i,k}} p(t,y) - \sum_{i\in I_k}1_{\X_{i,k}}(x)\, p(x,y) \bigg|\, d\mu(y) \\
&\leq& \int_{\X_k} \sum_{i\in I_k} 1_{\X_{i,k}}(x)\, \big|p(x,y) - \inf_{t\in \X_{i,k}} p(t,y)\big| \, d\mu(y) \\
&\leq& \int_{\X_k} \sum_{i\in I_k} 1_{\X_{i,k}}(x) \sup_{u\in \X_{i,k}} \big|p(x,y) - p(u,y)\big| \, d\mu(y) \\
&\leq&  \delta_k^{\theta}\int_{\X_k} \sum_{i\in I_k} 1_{\X_{i,k}}(x)\,  L_{i,k,\theta}(y) \, d\mu(y) \\
&\leq&  \delta_k^{\theta} \sum_{i\in I_k} 1_{\X_{i,k}}(x)\int_{\X_k} \,  L_{i,k,\theta}(y) \, d\mu(y) \\
&\leq& \delta_k^{\theta} \, \ell_{k,\theta}. 
\end{eqnarray*}
We have proved that, for every $f\in{\cal B}_0$ such that $\|f\|_0\leq 1$ and for every $x\in\X_k$, we have 
\begin{equation} \label{ineq-cont-1}
\frac{\big|(\Qc_kf)(x)  - (Pf)(x)\big|}{V(x)} \leq \alpha_k +  \ell_{k,\theta} \, \delta_k^{\theta}.
\end{equation}
Moreover we deduce from the definition of $\psi_k$ and from (\ref{ineq-cont-1}) that  
\begin{equation} \label{ineq-cont-2}
0\leq \frac{\psi_k(x)}{V(x)} = \frac{1 - (\Qc_k1_{\X})(x)}{V(x)} = \frac{(P1_{\X})(x) - (\Qc_k1_{\X})(x)}{V(x)} \leq \alpha_k +  \ell_{k,\theta} \, \delta_k^{\theta}.
\end{equation}
It follows from Inequalities (\ref{ineq-cont-1}) and (\ref{ineq-cont-2}) that, for every $f\in{\cal B}_0$ such that $\|f\|_0\leq 1$ and for every $x\in\X_k$, we have: 
\begin{eqnarray*}
\frac{\big|(\Pc_kf)(x)  - (Pf)(x)\big|}{V(x)} 
&=&  \frac{\big|(\Qc_kf)(x) +f(x_0)\psi_k(x)  - (Pf)(x)\big|}{V(x)}  \\
&\leq& \frac{\psi_k(x)}{V(x)} + \frac{\big|(\Qc_kf)(x)  - (Pf)(x)\big|}{V(x)} \\
&\leq& 2\big(\alpha_k +  \ell_{k,\theta} \, \delta_k^{\theta}\big).
\end{eqnarray*}
This inequality and (\ref{ineg-1}) provide the conclusion of Lemma~\ref{lem-appro}. 
\end{proof}
\begin{arem} \label{rem-cste}
The inequality~$(b)$ of \cite[Prop.~2.1]{HerLed14} provides explicit bounds in (\ref{estim-VT}) and (\ref{estim-L1}) in terms of the constants $\delta$, $L$ in (\ref{drift}) and the constants $C$ and $\rho$ in (\ref {V-geo}). Unfortunately, finding explicit constants $\rho\in(0,1)$ and $C>0$ in (\ref{V-geo}) is a difficult issue, even for simple models as AR(1). Such constants can be obtained in our context by applying the procedure of \cite[Th.~4.1]{HerLed14}, but the resulting constant $C$  is too large to be numerically interesting. An alternative way is to use, for $k$ larger and larger, the bound provided by Inequality $(a)$ of \cite[Prop.~2.1(a)]{HerLed14}, which is only based on the spectral properties of the finite stochastic matrix $B_k$. But again the resulting constants are too large. In fact the numerical applications  presented in Section~\ref{sec-examp} show that the convergence in (\ref{estim-VT}) and (\ref{estim-L1}) is much better than what is provided by using the constants derived from \cite[Prop.~2.1]{HerLed14}. 
\end{arem}
%
\section{Applications to first order autoregressive processes} \label{sec-AR}

\subsection{The standard AR(1) process} \label{sub_AR1}
Let $(X_n)_{n\in\N}$ be a standard first order 
linear autoregressive process, that is 
\begin{equation} \label{ar1}
\forall n\geq1,\quad X_n = \varrho\, X_{n-1} + \vartheta_n,
\end{equation}
where $X_0$ is a real-valued random variable (r.v.) and $(\vartheta_n)_{n\in\N}$ is a sequence of real-valued independent and identically distributed (i.i.d.) random variables, also assumed to be independent from $X_0$. We suppose that $|\varrho|<1$, that $\vartheta_1$ has a pdf $\nu$, called the innovation density function, with respect to the Lebesgue measure $d\mu(y) := dy$ on $\R$. Assume that the three following conditions are satisfied: 
\vspace*{-3mm}
\begin{itemize}
	\item[(a)]  $\vartheta_1$ has a moment of order $m$ for some $m\in[1,+\infty)$, namely  
\begin{equation} \label{m}
\exists\, m\in[1,+\infty),\quad \eta_m := \int_\R |x|^m\, \nu(x)\, dx < \infty;
\end{equation}
\item[(b)] $\nu$ is continuously differentiable on $\R$ and its derivative $\nu'$ is assumed to be right differentiable on $\R$; 
\item [(c)] finally
\begin{equation} \label{lip}
   \ I':=\int_\R|\nu'(y)|\, dy < \infty \quad \text{and} \quad M'':=\sup_{t\in\R} |\nu_r''(t)| < \infty
\end{equation} 
where $\nu''_r(t)$ denotes the right derivative of $\nu'$ at $t$. 
\end{itemize}

Let $\X:=\R$ be equipped with its usual distance $d(x,x'):=|x-x'|$ and with its  Borel $\sigma$-algebra $\cX$. Recall that $(X_n)_{n\in\N}$ is a Markov chain with transition kernel $P$ defined by 
\begin{equation} \label{AR-kernel} 
\forall A\in\cX,\quad P(x,A) := \int_\R 1_A(y)\, p(x,y)\, dy \qquad \text{with }\ p(x,y) := \nu(y-\varrho x).
\end{equation} 
It is well-known from \cite{MeyTwe93} that $(X_n)_{n\in\N}$ admits a unique stationary probability measure $\pi$ on $\R$, and that $\pi$ is absolutely continuous with respect to the Lebesgue measure, with density function $\mathfrak{p}$ such that $\int_\R |y|^m\mathfrak{p}(y) dy <\infty$ and satisfying  
$$\forall x\in\R,\quad \mathfrak{p}(x) = \int_\R \nu(x-\varrho u)\, \mathfrak{p}(u)\, du.$$
For $x\in\R$, we define $V(x) = \lfloor 1+ |x|^m\rfloor$ where $m$ is the positive real number given in (\ref{m}) and where $\lfloor\cdot\rfloor$ denotes the integer part function on $\R$. According to (\ref{m}) and $|\varrho| < 1$, $P$ is $V$-geometrically ergodic (see \cite{MeyTwe93}). In the sequel we fix 
any $\delta\in (|\varrho|^m,1)$. It can be easily deduced from (\ref{m}) that there exists $M\equiv M(\delta)$ such that 
\begin{equation} \label{drift-AR}
PV \leq \delta V + M.
\end{equation}
For every $k \geq 1$, we choose $\delta_k>0$ such that $\delta_k = \text{O}(1/k)$ and (for the sake of simplicity) such that $q_k := 2k/\delta_k\in\N$. Set 
 $\ \X_k=[-k,k[$, and consider the following partition of $\X_k$: 
\begin{equation} \label{X-i-k-ar}
\X_k := 
\bigsqcup_{i=0}^{q_k-1} \X_{i,k}\qquad \text{with }\ 
\X_{i,k} := \big[x_{i,k},x_{i+1,k}\big[,\quad x_{i,k} = -k + i\, \delta_k.
\end{equation}
The associated discretized Markov kernels $\Pc_k$ and the  probability measures $\widehat\pi_k$ on $\R$ are defined by (\ref{def-hatPk}) and (\ref{def-hat-pik}) respectively. The associated function 
$\mathfrak{p}_k$ is given in (\ref{def-density-pk}).  
\begin{apro} \label{rate-ar}
Let $ \delta_k$ be such that $\delta_k>0$ and $\delta_k = \text{O}(1/k)$. Assume that the innovation density function $\nu(\cdot)$ satisfies Conditions~(\ref{m}) and (\ref{lip}). 
Then
\begin{equation} \label{erreur-AR}  
\|\pi-\widehat\pi_k\|_{TV} = 
\text{O}\big(|\ln \tau_k|\, \tau_k\big), \quad  \|\mathfrak{p} - \mathfrak{p}_k\|_{\L^1(\R)} = \text{O}\big(|\ln \tau_k|\, \tau_k\big) \quad 
\text{with} \ \tau_k =  \frac{1}{k^{m}}+\delta_k.
\end{equation}
\end{apro}
\begin{proof}{}
Proposition~\ref{rate-ar} follows from Theorem~\ref{theo-main} and Corollary~\ref{cor-main}, provided that Assumptions~(\ref{def-alpha}) and (\ref{Lk-finite}) are satisfied (all the others assumptions of Theorem~\ref{theo-main} have been already checked above). First the real number $\alpha_k$ in (\ref{def-alpha}) satisfies 
$$\alpha_k =  
\int_{|y| > k(1-|\varrho|)} \nu(y)\, dy \ \leq \frac{\eta_m}{(1-|\varrho|)^m\, k^m}$$
from Markov's inequality. Thus (\ref{def-alpha}) holds. Second, we obtain for every $(x,x')\in\X_{i,k}$
\begin{eqnarray*}
|p(x,y) - p(x',y)| &=& \big| \nu(y-\varrho x) - \nu(y-\varrho x')\big| \\
 &\leq & 
  |\varrho|\, |x-x'|\, \big| \nu'(y - \varrho \, c) \big| \quad \text{for some }\ c\equiv c_{x,x',y}\in\X_{i,k} \\
&\leq & 
 |\varrho|\, |x-x'|\,  \left(\big| \nu'(y - \varrho \, c) - 
\nu'(y - \varrho \, x_{i,k}) \big| + \big|\nu'(y - \varrho \, x_{i,k})\big|
 \right)  \\
 &\leq &  
|\varrho|\, |x-x'|\, \left(  
|\varrho|\, M''\,  \delta_k 
 + \left|\nu'(y - \varrho \,  x_{i,k})\right|\right)
\end{eqnarray*}
so that 
\begin{eqnarray*}
L_{i,k,1}(y) &:=& \sup\left\{\frac{|p(x,y) - p(x',y)|}{|x-x'|},\ (x,x')\in\X_{i,k}\times\X_{i,k},\ x\neq x'\right\} \\
&\leq& 
|\varrho|\left(|\varrho|\, M''\,  \delta_k  + \left|\nu'(y - \varrho \, x_{i,k})\right|\right). 
 \end{eqnarray*}
Using the notations of (\ref{lip}), we obtain that 
\begin{equation} \label{ell-k-1}
\ell_{k,1} := \max_{i\in I_k} 
\int_{-k}^k L_{i,k,\theta}(y) \, dy \leq2\, |\varrho|^2\, M''
 k\, \delta_k  + |\varrho|\, I'. 
\end{equation}
 Recall that $\delta_k = \text{O}(1/k)$ by hypothesis, so that $\sup_{k\geq 1} \ell_{k,1} < \infty$.
Hence (\ref{Lk-finite}) holds. 
\end{proof}
\begin{arem} 
Alternative assumptions (instead of (\ref{lip})) on the innovation density function $\nu$ are possible. For instance, in place of (\ref{lip}), we may suppose that the derivative $\nu'$ of $\nu$ exists and that $D:=\sup_{x\in\R} |\nu'(x)| < \infty$. Then $L_{i,k,1}(\cdot) \leq D$, so that $\ell_{k,1} = \text{O}(k)$. Thus, provided that $\delta_k = \text{o}(1/k)$, the statements of Proposition~\ref{rate-ar} is replaced with the following ones: 
$\|\pi-\widehat\pi_k\|_{TV}$ and $\|\mathfrak{p} - \mathfrak{p}_k\|_{\L^1(\R)}$ are both 
 $\text{O}\big(|\ln \tau_k|\, \tau_k\big)$ with $\tau_k = 1/k^m + k\, \delta_k$ and both $\|\pi-\widehat\pi_k\|_{TV}$ and $\|\mathfrak{p} - \mathfrak{p}_k\|_{\L^1(\R)}$ converge to $0$ when $k\r+\infty$ .
\end{arem}
\begin{arem}
 In \cite{DonGooMidRae00} a similar state-discretization procedure is proposed to estimate the spectrum of the Markov kernel $P$ given in (\ref{AR-kernel}). Because the authors of \cite{DonGooMidRae00} use the standard perturbation theory, they have to assume that the innovation density function is compactly supported in some interval $[a,b]$ (the action of $P$ is then considered on the usual Lebesgue space $\L^2([a,b])$.  The use of the Keller-Liverani perturbation theorem in our work (see the proof of Theorem~\ref{theo-main}) allows us to consider innovation density functions with unbounded support. 
\end{arem}
\begin{arem} \label{moins-reg}
If $(X_n)_{n\in\N}$ is a first order autoregressive model given by (\ref{ar1}), then for any $\ell\geq1$ the sequence $(X^{(\ell)}_n)_{n\geq0}$ defined by $X^{(\ell)}_n := X_{\ell n}$ satisfies the following linear recursion 
\begin{equation} \label{AR_itere}
\forall n\geq1,\quad X^{(\ell)}_n = \varrho^\ell\, X^{(\ell)}_{n-1} + \vartheta^{\, (\ell)}_n \quad \text{with } \vartheta^{\, (\ell)}_n = \sum_{k\, =\, \ell(n-1)+1}^{\ell\, n} \varrho^{\ell n-k}\, \vartheta_k.
\end{equation}
The sequence $(\vartheta^{\, (\ell)}_n)_{n\geq0}$ is i.i.d., and  $(X^{(\ell)}_n)_{n\geq0}$ is a first order autoregressive model having the same stationary density function as  $(X_n)_{n\in\N}$. The transition kernel of $(X^{(\ell)}_n)_{n\geq0}$ is $P^\ell$, which is of the form (\ref{form-P}) too. More precisely, for every $x\in\R$ we have  $P^{\ell}(x,dy) = p_\ell(x,y)\, dy$ with $p_\ell(x,y) := \nu_\ell(y-\varrho^\ell x)$, where 
$\nu_\ell$ is the pdf of $\vartheta^{\, (\ell)}_1$, that is
$\nu_\ell := \mu_{\ell}\star\cdots\star\mu_{1}$, where $\mu_k$ denotes the pdf of the r.v.~$\varrho^{\ell-k}\, \vartheta_k$ for $k=1,\ldots,\ell$, and where the symbol "$\star$" stands for the standard convolution product. This fact may be relevant since $\nu_\ell$ is more and more regular as $\ell$ increases, so that  $\nu_\ell$ may satisfy the regularity condition required in (\ref{lip}) for $\ell$ large enough. In this case the stationary density function of $(X_n)_{n\in\N}$ can be approximated by applying Proposition~\ref{rate-ar} to $P^\ell$ (thus with $\nu_\ell$ in place of $\nu$). 
For instance, if the innovation law is the uniform distribution on $[0,1]$, then Proposition~\ref{rate-ar} applies to the Markov kernel $P^3$ since the associated innovation density function (i.e.~the law of $\vartheta^{\, (3)}_1 := \varrho^{2}\, \vartheta_1 + \varrho\, \vartheta_2 +  \vartheta_3$) is continuously differentiable on $\R$ and satisfies~(\ref{lip}). 
\end{arem}
\subsection{The AR(1) process with ARCH$(1)$ errors} \label{sec-arch} 
The following example is derived from \cite{BorKlu01}. Let $\X:=\R$ be equipped with its usual distance $d(x,x'):=|x-x'|$ and with its  Borel $\sigma$-algebra $\cX$. Let $\alpha\in\R$ and let $\beta,\lambda>0$. We consider the autoregressive process $(X_n)_{n\in\N}$ with ARCH$(1)$ errors, defined by  
\begin{equation} \label{arch}
\forall n\geq1,\quad X_n = \alpha\, X_{n-1} + \big(\beta+\lambda X_{n-1}^2\big)^{1/2}\, \vartheta_n,
\end{equation}
where $X_0$ is a real-valued r.v.~and $(\vartheta_n)_{n\in\N}$ is a sequence of  i.i.d. real-valued random variables which are independent from $X_0$. We suppose that $\vartheta_1$ has a pdf $\nu$ with respect to the Lebesgue measure $d\mu(y) := dy$ on $\R$, that $\nu$ is a bounded continuously differentiable and symmetric function with full support $\R$, that its derivatives $\nu'$ satisfies $|\nu'(x)| = \text{O}_{\pm\infty}(1/|x|)$, and finally that $\vartheta_1$ has a second-order moment. Then $(X_n)_{n\in\N}$ is a Markov chain with transition kernel $P$ defined by $P(x,A) := \int_\R 1_A(y)\, p(x,y)\, dy\ $ ($A\in\cX$) with 
\begin{equation} \label{AR-arch} 
p(x,y) := \big(\beta+\lambda x^2\big)^{-1/2}\, \nu\left(\frac{y-\alpha x}{\big(\beta+\lambda x^2\big)^{1/2}}\right).
\end{equation} 
As in Section~\ref{sec-AR}, for every $k\geq 1$ we consider $q_k:=2k/\delta_k$, $\ \X_k=[-k,k[$, and $\X_{i,k}$ as in (\ref{X-i-k-ar}).  
Moreover assume that  
\begin{equation} \label{arch-cond-para} 
\E\big[\ln\big|\alpha + \sqrt{\lambda} \, \vartheta_1\big|\, \big] < 0.
\end{equation} 
Then there exists $\kappa > 0$ such that, for every $u\in]0,\kappa[$, we have  
$\E[|\alpha + \sqrt{\lambda} \, \vartheta_1|^u] < 1$, see \cite[Prop.~2]{BorKlu01}. Let $\eta\in ]0,\min(\kappa,2)[$.
\begin{apro} \label{pro-arch}
Under the previous assumptions and notations, the following estimates hold true:
\begin{equation} \label{erreur-arch}  
\|\pi-\widehat\pi_k\|_{TV} = 
\text{O}\big(|\ln \tau_k|\, \tau_k\big), \quad  \|\mathfrak{p} - \mathfrak{p}_k\|_{\L^1(\R)} = \text{O}\big(|\ln \tau_k|\, \tau_k\big) \quad 
\text{with} \ \tau_k =  
\frac{1}{k^{\eta/2}}
+ k\, \delta_k. 
\end{equation}
Thus $\|\pi-\widehat\pi_k\|_{TV}$ and $\|\mathfrak{p} - \mathfrak{p}_k\|_{\L^1(\R)}$ converge to $0$ when $k\r+\infty$ provided that $\delta_k = \text{o}(1/k)$.
\end{apro}
\begin{proof}{}
For $x\in\R$, we define $V(x) = 1+ |x|^\eta$. The $V$-geometrical ergodicity of $P$, together with Condition (\ref{drift}), are proved in \cite[Th.~1]{BorKlu01}. To study (\ref{def-alpha}), we assume that $\alpha >0$ (similar arguments hold if $\alpha < 0$). Note that 
$$\int_{k}^{+\infty} p(x,y)\, dy = \int_{\phi_k(x)} \nu(t)\, dt\qquad \text{with}\quad \phi_k(x) := \frac{k-\alpha x}{\big(\beta+\lambda x^2\big)^{1/2}}.$$
Let $x\in[-k,k]$. If $\sqrt{k} \leq |x| \leq k$, then 
$$\frac{1}{1+|x|^\eta} 
\int_{k}^{+\infty} p(x,y)\, dy \leq 
\frac{1}{1+k^{\eta/2}}.$$
If $|x| \leq \sqrt{k}$, then $\frac{1}{1+|x|^\eta} \leq 1$ and
$$\int_{\phi_k(x)} \nu(t)\, dt \leq \int_{\phi_k(\sqrt{k})}\nu(t)\, dt \ = \text{O}\big(1/k\big)$$
from $\phi_k(\sqrt{k}) \leq \phi_k(x)$, $\phi_k(\sqrt{k})\sim_{+\infty} (k/\lambda)^{1/2}$, and 
from Markov's inequality (since by hypothesis $\nu$ has a second-order moment). 
Since $\eta < 2$, we have proved that  
$$\sup_{|x| \leq k} \int_{k}^{+\infty} p(x,y)\, dy = 
\text{O}\left(\frac{1}{k^{\eta/2}}\right).$$ 
The same conclusion can be similarly obtained for the term $\sup_{|x| \leq k}\int_{-\infty}^{-k} p(x,y)\, dy$.
Consequently  $\alpha_k$ in (\ref{def-alpha}) satisfies: 
$\alpha_k = \text{O}\big(k^{-\eta/2}\big)$.
Next, to obtain (\ref{Lk-finite}) set $M:=\sup_{x\in \R} \nu(x)$, $M':=\sup_{x\in \R} \nu'(x)$, and $C := \sup_{x\in \R} |x\, \nu'(x)|$. An easy computation gives  
\begin{eqnarray*}
\left|\frac{\partial p}{\partial x}(x,y)\right| &\leq& \frac{M\lambda|x|}
{(\beta+\lambda x^2)^{3/2}} + \frac{M'|\alpha|}
{\beta+\lambda x^2} + \frac{C\lambda|x|}{(\beta+\lambda x^2)^{3/2}}.
\end{eqnarray*}
Thus $D := \sup_{(x,y)\in\R^2}|\frac{\partial p}{\partial x}(x,y)| < \infty$, so that the function $L_{i,k,\theta}$ defined in (\ref{L-i-k-theta}) satisfies (with $\theta=1$): 
$\forall y\in\R,\ L_{i,k,1}(y) \leq D$. Therefore the real numbers $\ell_{k,\theta}$ in (\ref{Lk-finite}) are such that $\ell_{k,1} \leq 2Dk$. The above inequalities and Theorem~\ref{theo-main} provide the desired statement in Proposition~\ref{pro-arch}.  
\end{proof}

\section{A generic algorithm to get $\mathfrak{p}_k(y)$} \label{ap-algo-ar}
Let $(X_n)_{n\in\N}$ be a Markov chain with transition kernel $p(\cdot,\cdot)$. In this section, we propose a generic algorithm to get the material provided by Section~\ref{sect-disc}. Specifically, the focus is on the non-negative function $\mathfrak{p}_k$ (\ref{def-density-pk}) which allows us to obtain the approximating invariant probability given by Proposition \ref{def_approx}. According to Section~\ref{sect-disc}, the following algorithm can be proposed.  
\leftmargini 1.2em
\begin{enumerate} \setlength{\itemsep}{-1mm}
	\item Fix the positive integer $k$ such that $\X_k :=[-k,k[$ and  choose the integers $k^{-}$ et $k^{+}$ such that $[k^{-},k^{+}[ \subset X_k$ (you can take $k^{-}=-k,k^{+}=k$). 
		\item Choose a mesh $\delta_k$ of the partition of $[k^{-},k^{+}[$  such that the number of intervals of the subdivision is $q_{\max} := (k^{+}-k^{-})/\delta_k \in \N^*$, 

	Let us introduce the ($q_{\max}+1)$ points of the subdivision $\big\{x_{i,k} := k^{-} + i \delta_k, i=0,\ldots, q_{\max}\big\}$ and consider the finite partition $\{\X_{i,k}=[x_{i,k},x_{i+1,k}[, i=0,\ldots, q_{\max}-1\}$ of $[k^{-},k^{+}[$
	
	\item Introduce $$p_{i,k}(y) := \inf_{t\in \X_{i,k}} p(t,y).$$
	\item  Choose $j_0\in\{1,\ldots,q_{\max}-1\}$, then for $j=j_0$ compute:  
\begin{eqnarray*}
B_k(i,j_0)\  &:=& \int_{x_{j_0,k}}^{x_{j_0+1,k}} p_{i,k}(y) \, dy +  1 -  \int_{k^{-}}^{k^{+}} p_{i,k}(y)  \, dy \qquad \text{for $i=0,\ldots,q_{\max}-1$} \\[0.15cm]
B_k(q_{\max},j_0) &:=&  1 
	\end{eqnarray*}
		 Compute for $j=0,\ldots,q_{\max}-1$, $j\neq j_0$,
\begin{eqnarray*}
B_k(i,j)\  &:=&  \int_{x_{j,k}}^{x_{j+1,k}} p_{i,k}(y)\, dy \qquad \text{for $i=0,\ldots,q_{\max}-1$} \\[0.15cm]	
B_k(q_{\max},j) &:=& 0 \qquad \text{pour $i=q_{\max}$}
\end{eqnarray*}
   Set $B(i,j) = 0$ for $j=q_{\max}$ et $i=1,\ldots,q_{\max}$.
  \item Compute the $B_k$-invariant probability vector $\pi_k$ of $B_k$: it has the form $\pi_k = (\{\pi_{i,k}\}_{0\leq i <q_{\max}}\, ,\, 0)$
  \item Finally, the non-negative function $\mathfrak{p}_k(\cdot)$ is defined by (see (\ref{def-density-pk})):
$$\forall y\in\R,\quad \mathfrak{p}_k(y) :=  1_{[k^{-},k^{+}[}(y)\sum_{i=0}^{q_{\max}-1} \pi_{i,k}\,  p_{i,k}(y) .$$

\end{enumerate}

The third step of the algorithm involves the computation of an extreme value of the function $t \mapsto p(t,y)$ on a small interval (length $\delta_k$). Such a numerical minimization may be computationally expensive. But it can be checked that, for AR(1) models in Subsections~\ref{AR1_gauss}, \ref{AR1_autres}, the function $t \mapsto p(t,y)$ has no local minima so that the minimum may be setted to $\min(p(x_{i,k},y),p(x_{i+1,k},y))$. The case of the AR(1) with ARCH(1) errors may produce local minima for some parameter $(\beta, \alpha,\lambda)$. But it can be expected that any approximation of $p_{i,k}(y)$ in Step 3.~does not provide large numerical errors from the fact that it is made on a very small interval of length $\delta_k<<1$.  

Such an algorithm has been implemented using MATLAB software to obtain the numerical results of Section \ref{sec-examp}.

\begin{arem}[Multivariate Markov models] A natural issue is the generalization of the material of Sections \ref{sect-disc} and \ref{sect-conv-gene} to multivariate Markov models. A general discussion is beyond the scope of this paper. We only mention that technical Conditions (\ref{drift},\ref{def-alpha},\ref{Lk-finite}) have natural counterparts for multivariate autoregressive models (e.g. see \cite{MeyTwe93}). Thus it can seen from this section that the main difficulties in a multidimensional framework are computational issues due to  computation of extreme values and integrals.
\end{arem}

\section{Numerical examples} \label{sec-examp}
\subsection{Application to the Gaussian AR(1)} \label{AR1_gauss} 
The benchmark model is the Gaussian linear model where the random variables $\vartheta_n$ in (\ref{ar1}) have Gaussian distribution $\mathcal{N}(0,\sigma^2)$. In such a context, it is well-known that the invariant probability $\pi$ of the Markov chain $(X_n)_{n \in \N}$ specified by (\ref{ar1}) is $\cN(0,\sigma^2/(1-\varrho^2))$. Therefore the pdf's $\nu$ and $\mathfrak{p}$ are 
	\begin{equation} \label{cas_gaussian}
	\nu(y)=\frac{1}{\sqrt{2\pi \sigma^2}} \exp\left(-\frac{y^2}{2 \sigma^2}\right) \qquad \mathfrak{p}(y)=\frac{\sqrt{1-\varrho^2}}{\sqrt{2\pi \sigma^2}} \exp\left(-\frac{(1-\varrho^2)y^2}{2 \sigma^2}\right).
\end{equation}
Using the algorithm in Section~\ref{ap-algo-ar}, we obtain the following numerical results. For the sake of simplicity, set $\sigma^2:=1$. The support of the approximation is $\X_k:=[-k,k[$ for specific value of the positive integer $k$, and $\delta_k$ is the mesh of the partition of $\X_k$ used for the computation. The supremum norm of the error vector $v_k:=\big(\mathfrak{p}_k(x_{i,k}) - \mathfrak{p}(x_{i,k})\big)_{i=0,\ldots,q_k}$  between $\mathfrak{p}_k$ and $\mathfrak{p}$ on the grid of points $\{x_{i,k},i=0,\ldots,q_k\}$ given by the partition of $\X_k$  (see (\ref{X-i-k-ar})) is denoted by $\|v_k\|_{\infty}$ and reported in Table~\ref{Table_gaussien}. The Riemann sum estimation $\|v_{k}\|_{1,R}:= \delta_k \|v_{k}\|_{1}$ of $\|\mathfrak{p}_k-\mathfrak{p}\|_{\L^1}$ is provided.  These errors are computed using a decreasing sequence of meshes $\delta_k$ and a support $\X_k$ selected according to the comments of Remark~\ref{Rem_support}.  As it can be seen, the quality of the approximation is quite satisfactory. Figure~\ref{Figure_gussien} gives the graphs of the two pdf $\mathfrak{p}_k$ and $\nu$. Note that the exact invariant pdf is not reported in Figure~\ref{Figure_gussien} since the estimated and exact graphs cannot be distinguished at this scale. From Table~\ref{Table_gaussien}, whatever the value of $\varrho$, the errors norms $\|v_{k}\|_{\infty}$ or $\|v_{k}\|_{1,R}$ scale linearly with the mesh $\delta_k$. 
\begin{table}[h]
\begin{center}
{\scriptsize \begin{tabular}{r|ccc|ccc|ccc|} \cline{2-10}
& \multicolumn{9}{c|}{$\boldsymbol\varrho$} \\ \cline{2-10} 
 & \multicolumn{3}{c|}{$\boldsymbol{0.5}$} & \multicolumn{3}{c|}{$\boldsymbol{0.7}$} & \multicolumn{3}{c|}{ $\boldsymbol{0.9}$} \\ \hline\hline
$\boldsymbol k$ & \multicolumn{3}{c|}{8}  
& \multicolumn{3}{c|}{14}&  \multicolumn{3}{c|}{40}\\ \hline
$\boldsymbol{\delta_k}$ & 0.05 & 0.02 & 0.005  
& 0.05 & 0.02 & 0.005 & 0.05 & 0.02 & 0.005\\ \hline
$\boldsymbol{\|v_k\|_{1,R}}$ & 0.01 & 0.004 &  $0.001$  
& 0.0151 &  0.0061 & 0.0015  & 0.0540 & 0.025 & 0.0058\\ \hline
$\boldsymbol{\|v_k\|_{\infty}}$ & 0.0025 & 0.001 &  $2.45\times 10^{-04}$  
& 0.0035 & 0.0014 & $3.45\times 10^{-4}$ & 0.0099 & 0.0041 & 0.0011\\ \hline
\end{tabular}}
\end{center}
\vspace*{-3mm}
\caption{Numerical results for the Gaussian linear model}
\label{Table_gaussien}
\end{table}

\begin{figure}[h]
\centering
\includegraphics[scale=0.6]{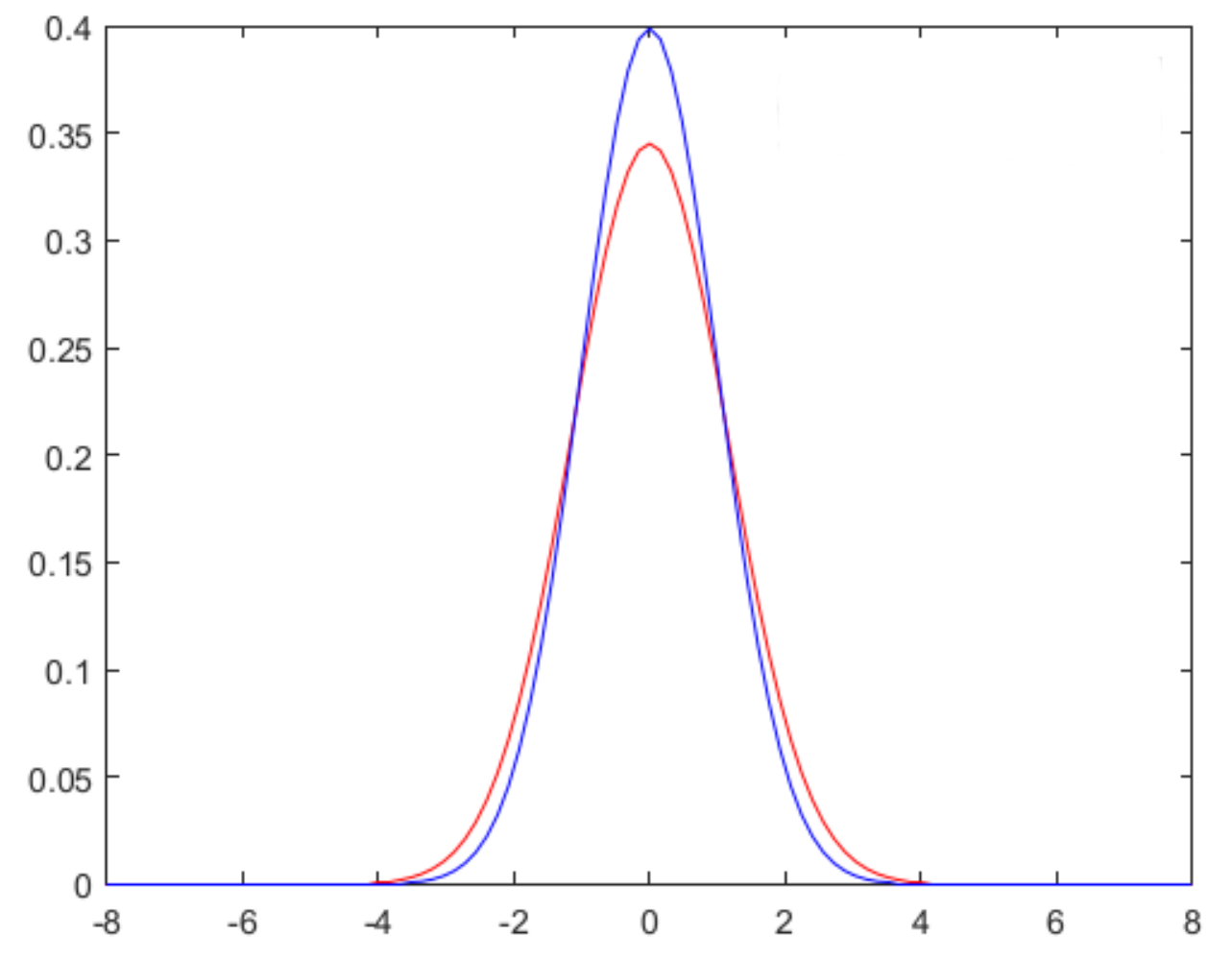} 

\includegraphics[scale=0.6]{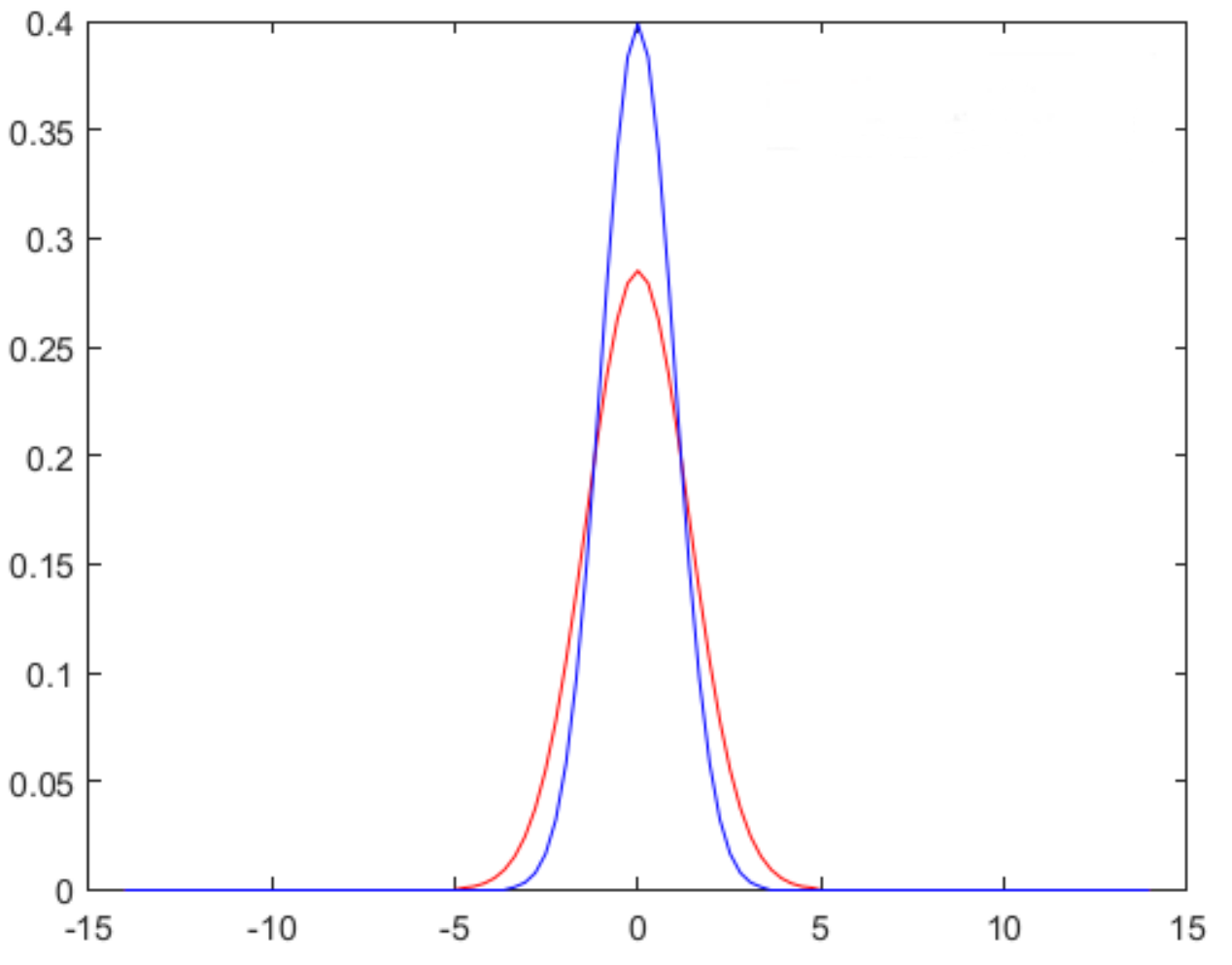} 

\includegraphics[scale=0.6]{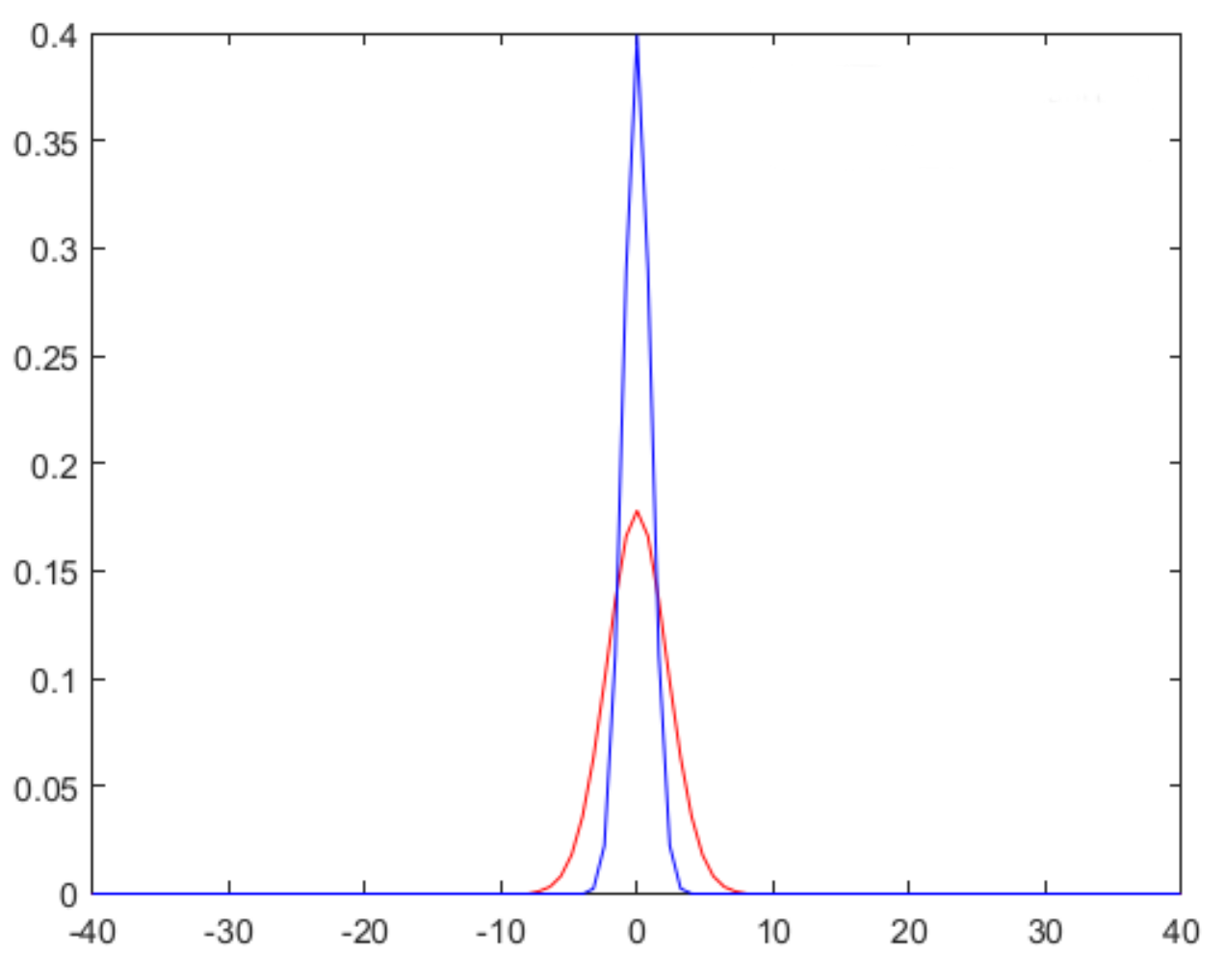}
\caption{(Gauss) For $\varrho=0.5, 0.7, 0.9$ and $\delta_{k}=0.02$: graphs  of the invariant pdf $\mathfrak{p}_{8},\mathfrak{p}_{14},\mathfrak{p}_{40}$ (red) and the innovation pdf $\nu$ (blue)}
\label{Figure_gussien}
\end{figure}

\begin{arem} \label{Rem_support}  The algorithm is sensitive to the support $\X_k$ of the approximate function $\mathfrak{p}_k$. Indeed, if the value of $k$ in $\X_k$ is too small with respect to the support of the target pdf $\mathfrak{p}$ of $\pi$, then the approximate function $\mathfrak{p}_k$ may appear to be far from the target $\mathfrak{p}$.
 
\end{arem}

\subsection{Applications to AR(1) where the invariant pdf $\mathfrak{p}$ is unknown} \label{AR1_autres}

When the target pdf $\mathfrak{p}$ is unknown, the set $\X_k$ (the support of $\mathfrak{p}_k$, see  (\ref{X-i-k-ar})) can be chosen as follows. If the innovation pdf $\nu$ has a support contained in $[-a,a]$  and if $X_0=0$, then $\P\{X_n \in[-s_n,s_n]\}=1$ with $s_n:=a\sum_{k=0}^{n-1} \varrho^k$, so that the pdf $\mathfrak{p}$ has support $[-s,s]$ with $s:=a/(1-\varrho)$. If $\nu$ is not compactly supported, then the previous remark may be applied with $a$ such that $\nu(x)$ is meaningless for $|x|>a$. Obviously this remark may be easily adapted when the exact or approximated support of $\nu$ is contained in $[0,a]$. In the previous Gaussian case, although this question is less relevant since the target pdf $\mathfrak{p}$ is known, we take $a=4$ and $s=4/(1-\varrho)$ in Figure~\ref{Figure_gussien}.  

\subsubsection{Exponential innovation distribution}
In this part, the innovation distribution $\nu$ is set to the exponential one with parameter $1$. Recall that the pdf $\nu$ must satisfy the regularity conditions of Proposition~\ref{rate-ar}. Therefore, as discussed in Remark~\ref{moins-reg}, the pdf $\nu_3$ is used as input in the algorithm instead of $\nu$: 
$$\nu_{3,\varrho}(x) = \frac{1}{(1-\varrho)^2} \bigg(e^{-x} - e^{-x/\varrho}  + \frac{\varrho\, \big(e^{-x/\varrho^2} - e^{-x}\big)}{1+\varrho}\bigg)\, 1_{[0,+\infty[}(x)$$
and the dynamics is given by (\ref{AR_itere}) with $\ell :=3$.
The support of $\nu_{3,\varrho}$ may be truncated to $[0,a_{3,\varrho}]$ with $a_{3,0.5}=11,a_{3,0.7}=12, a_{3,0.9}=14$, so that the support of $\mathfrak{p}$ may be truncated to $[0,s_{\varrho}]$ with $s_{\varrho}=\lfloor a_{3,\varrho}/(1-\varrho^3)\rfloor +1$, that is $s_{0.5}=13,s_{0.7}=19,s_{0.9}=52$.  
Thus we use the interval $[0,13], [0,19], [0,52]$ as $\X_k$ for $\varrho:=0.5,0.7,0.9$ (apply the above remark with $\nu_{3,\varrho}$ and $\varrho^3$ in place of $\nu$ and $\varrho$).   
In Figure~\ref{Expo_rho05} are reported the graphs of the estimated $\mathfrak{p}_k$ of the (unknown) invariant pdf $\mathfrak{p}$ for $\varrho=0.5, 0.7, 0.9$. 
\begin{figure}[h]
	\centering
		\includegraphics[scale=0.6]{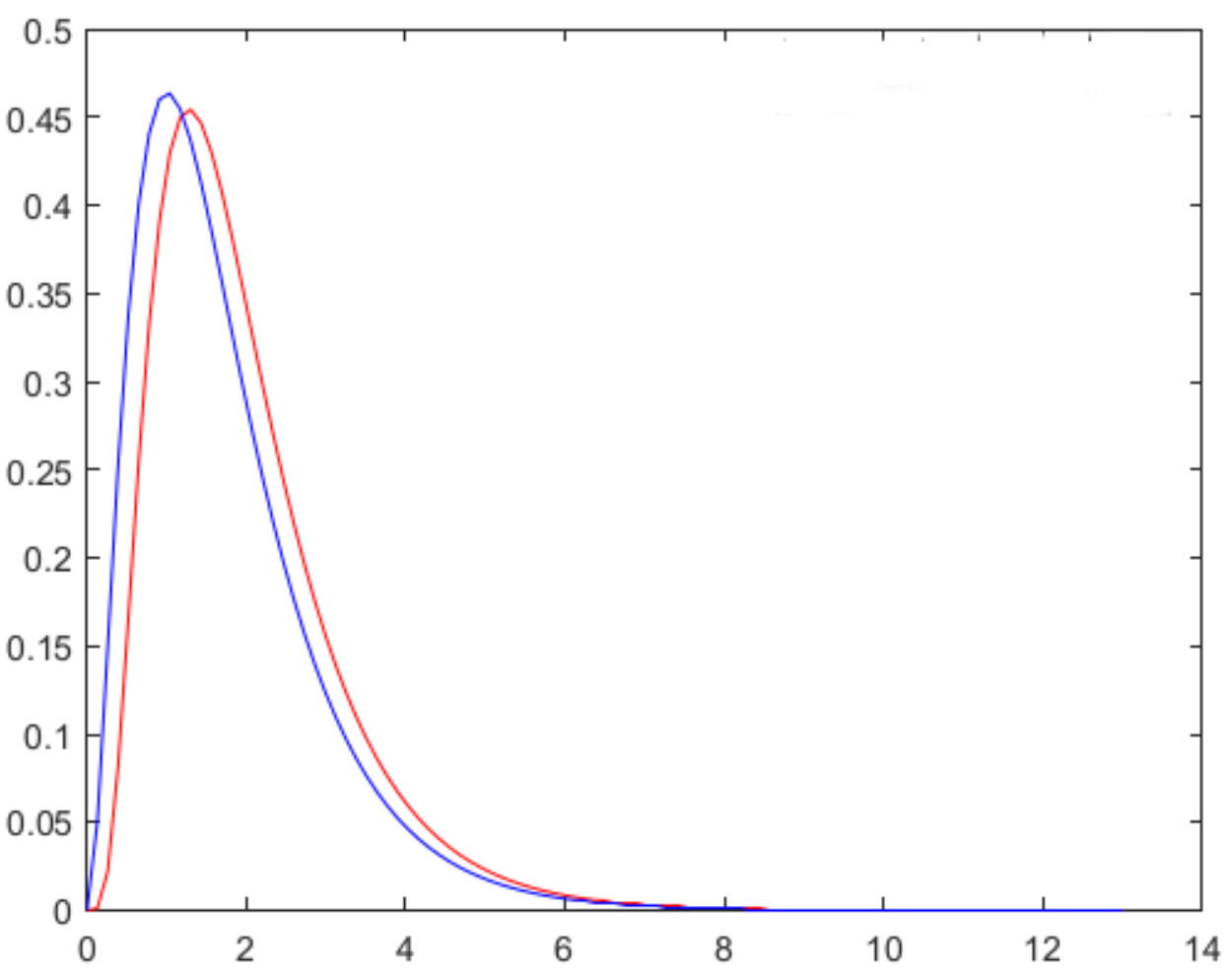} 
		
		\includegraphics[scale=0.6]{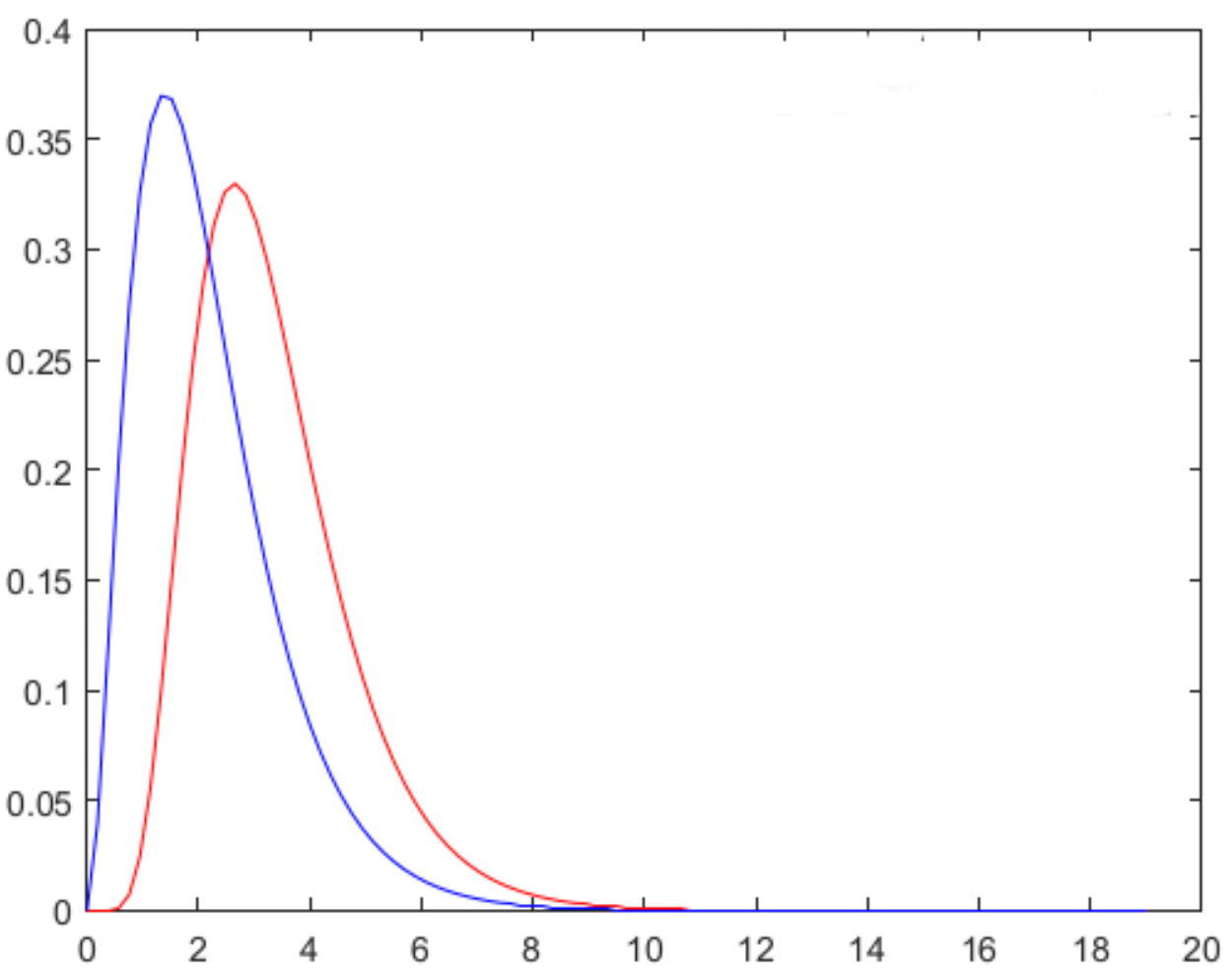}  
		
		\includegraphics[scale=0.6]{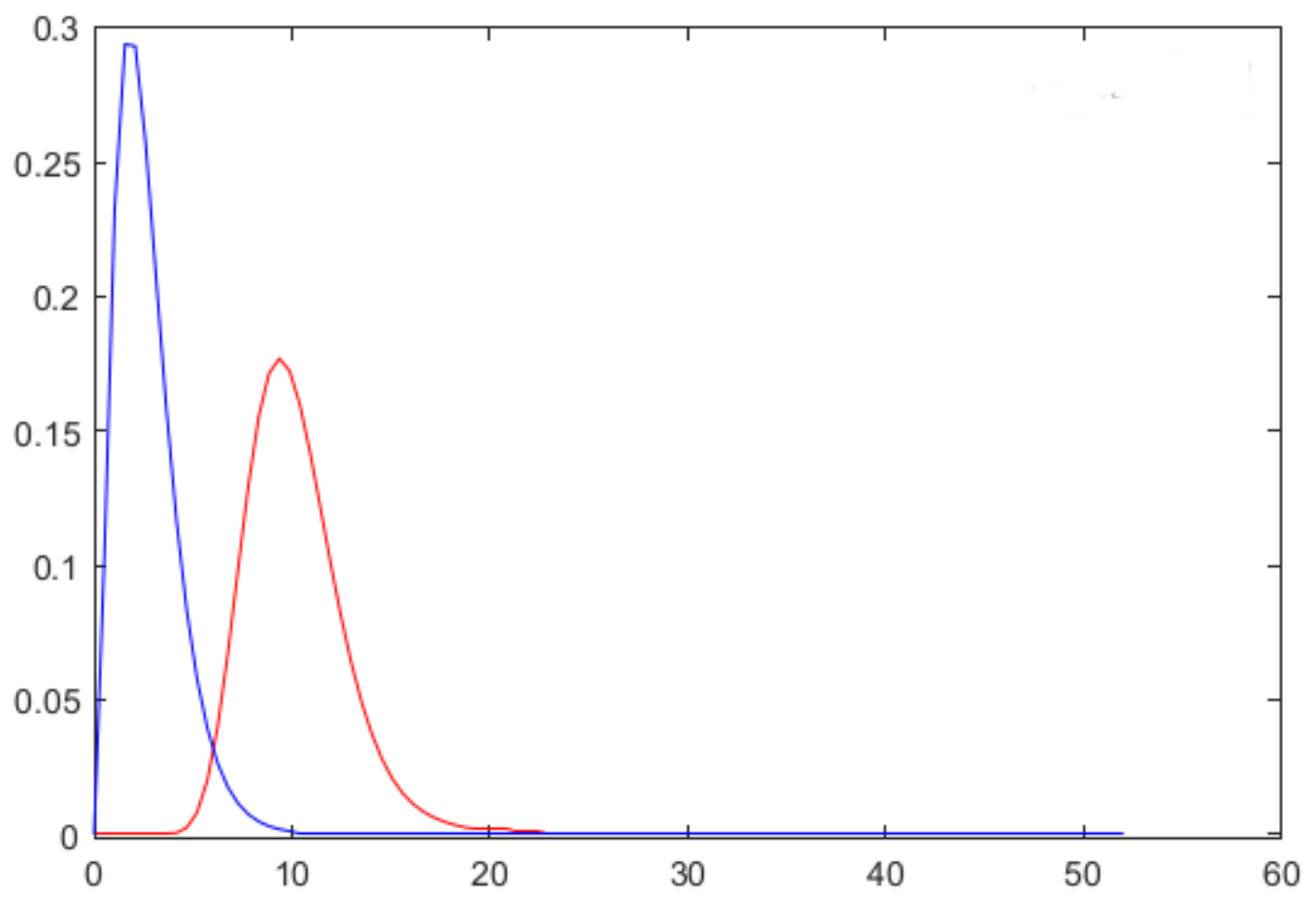}
\caption{(Expo) For $\varrho=0.5,0.7,0.9$ with $\delta_{k}=0.02$: graphs of the estimated invariant pdf $\mathfrak{p}_{k}$ (red) and the pdf $\nu_{3,\varrho}$ (blue)}
\label{Expo_rho05}
\end{figure}

The invariant probability distribution $\pi$ with pdf $\mathfrak{p}$ satisfies $\pi P = \pi$, that is: $\int_{\R} \mathfrak{p}(x) \, p(x,\cdot) dx = \mathfrak{p}(\cdot)$. Such a relation can be checked on the grid of points $\{x_{i,k},i=0,\ldots,q_k\}$ given by the partition of $\X_k$ (see (\ref{X-i-k-ar})): 
\[ \forall x_{i,k}, \quad \int_{\R} \mathfrak{p}(x) \, p(x,x_{i,k}) dx = \mathfrak{p}(x_{i,k}).\] 
The integral on the left hand side can be estimated using the Riemann sum denoted by $\widetilde{\mathfrak{p}}_k(x_{i,k}):=\sum_j \mathfrak{p}_k(x_{j,k}) P(x_{j,k},x_{i,k}) \delta_k$. Therefore, in order to get some confidence into the estimated invariant pdf $\mathfrak{p}_k$,  the uniform norm of the following vector $ w_k:=\big(\mathfrak{p}_k(x_{i,k}) - \widetilde{\mathfrak{p}}_k(x_{i,k})\big)_{i=0,\ldots,q_k}$ is  reported in Table~\ref{Table_P-invariance}. As it can be seen, the results are  satisfactory.

\begin{arem}
The case $\varrho=0.9$ (even $\varrho=0.8$) shows that the graphs of $\nu_{3,\varrho}$ and $\mathfrak{p}$ (given by the approximation $\mathfrak{p}_k$) are very far. Consequently, in this case, the method of \cite{Hai98,AndHra00,AndHra07}  requires to compute $h_N$ via (\ref{eq-hn}) for some quite large integer $N$. Since the use of (\ref{eq-hn}) is recursive, the successive approximations $h_1,\ldots,h_N$ should involve large cumulative errors. A similar comment  holds true in the forthcoming case of the uniform innovation distribution. As mentioned in Introduction, our method does not contain this drawback since it is not based on a recursive algorithm. 
\end{arem}

\subsubsection{Uniform innovation distribution}

Here, the innovation distribution $\nu$ is set to the uniform one on $[0,1]$. 
As discussed in Remark~\ref{moins-reg}, the pdf $\nu_3$ is used as input in the algorithm instead of $\nu$ (see \cite[p 281-282]{AndHra07} for an explicit formula). 
Here the dynamics is given by (\ref{AR_itere}) with $\ell :=3$. The graphs of $\nu_{3,\varrho}$ with $\varrho = 0.4, 0.5, 0.6, 0.7, 0.8, 0.9$ are reported in Figures~\ref{Unif_rho04}, \ref{Unif_rho07} (blue curves). 
The support of $\nu_{3,\varrho}$ is $[0,1+\varrho+\varrho^2]$ so that the support of the target pdf $\mathfrak{p_{\varrho}}$ is included into $[0,s_{\varrho}]$ with $s_{\varrho}:=\lfloor(1+\varrho+\varrho^2)/(1-\varrho^3) \rfloor+1 $. Thus we use the intervals $[0,2], [0,2], [0,3], [0,4], [0,5], [0,10]$ as set $\X_k$ for $\varrho:=0.4,0.5,0.6,0.7,0.8,0.9$.    
In  Figure~\ref{Unif_rho04}, we report the graphs of the approximated function $\mathfrak{p}_k$ of the (unknown) invariant pdf $\mathfrak{p}$ and the pdf $\nu_{3,\varrho}$ for $\varrho=0.4,0.5,0.6$.  The graphs for $\varrho=0.7,0.8,0.9$ are reported in  Figure~\ref{Unif_rho07}. As in the exponential case, the expected invariance of the estimated pdf $\mathfrak{p}_k$ is evaluated by the uniform norm of the following vector $ w_k:=\big(\mathfrak{p}_k(x_{i,k}) - \widetilde{\mathfrak{p}}_k(x_{i,k})\big)_{i=0,\ldots,q_k}$ (see Table~\ref{Table_P-invariance}). The results are still satisfactory. 

\begin{table}[h]
\begin{center}
{\scriptsize \begin{tabular}{r|ccc|} \cline{2-4}
& \multicolumn{3}{c|}{\text{Expo}} \\ \cline{1-4}   
$\varrho$ & 0.5 & 0.7 & 0.9 \\ \hline
$\|w_k\|_{\infty}$ & $8.73\times 10^{-4}$ & $9.47 \times 10^{-4}$ &  0.0013 \\\hline
\end{tabular}
\hspace*{5mm} 
\begin{tabular}{r|ccc|} \cline{2-4}
& \multicolumn{3}{c|}{\text{Unif}} \\ \cline{1-4}   
$\varrho$ & 0.7 & 0.8 & 0.9 \\ \hline
$\|w_k\|_{\infty}$ & 0.0045 & 0.0050 &  0.0051 \\\hline
\end{tabular} 
}
\end{center}
\vspace*{-3mm}
\caption{AR(1): checking $P$-invariance of the estimated pdf $\mathfrak{p}_k$ with $\delta_k:=0.02$}
\label{Table_P-invariance}
\end{table}

\begin{figure}[h]
	\centering
		\includegraphics[scale=0.6]{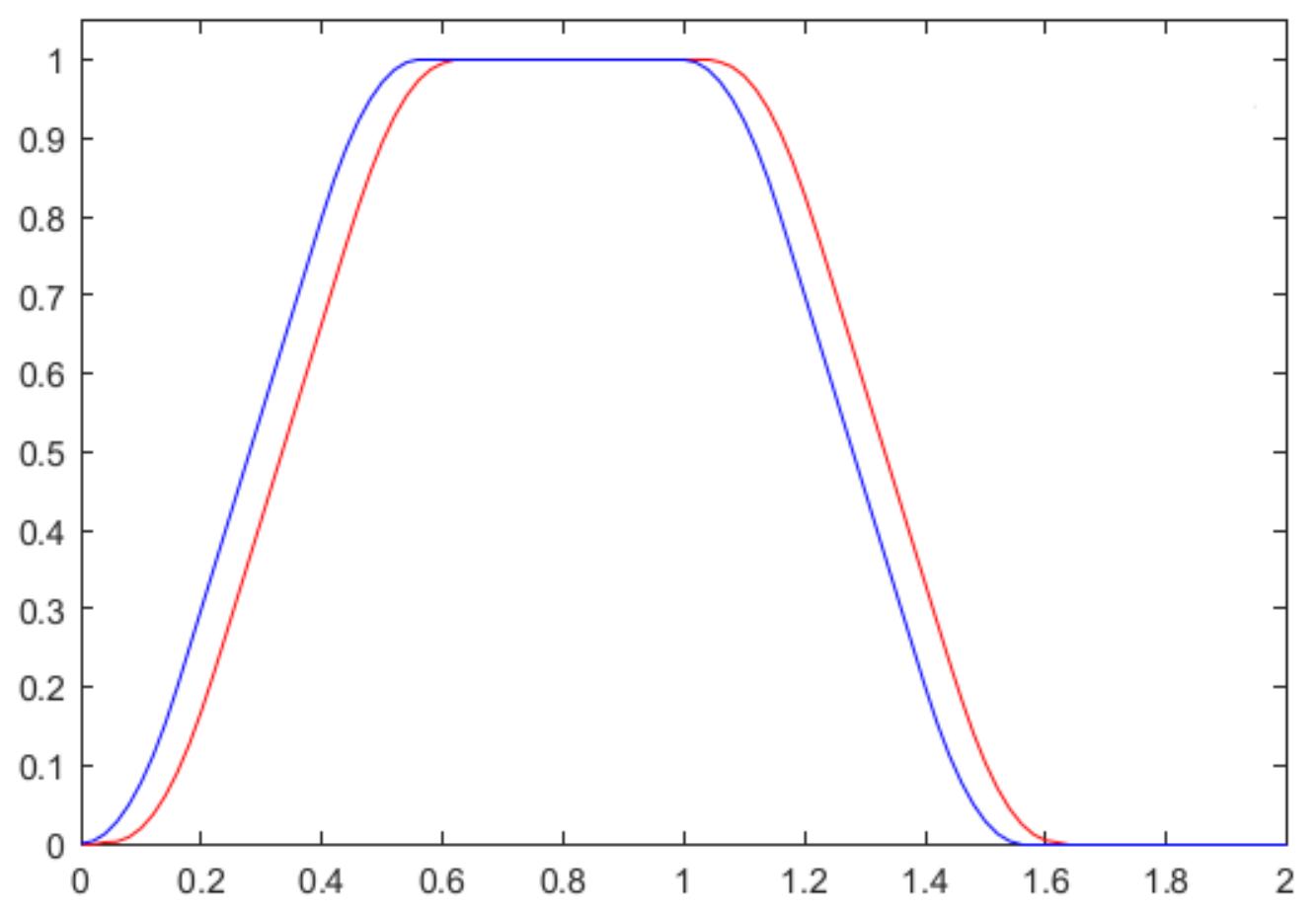}	
		
		\includegraphics[scale=0.6]{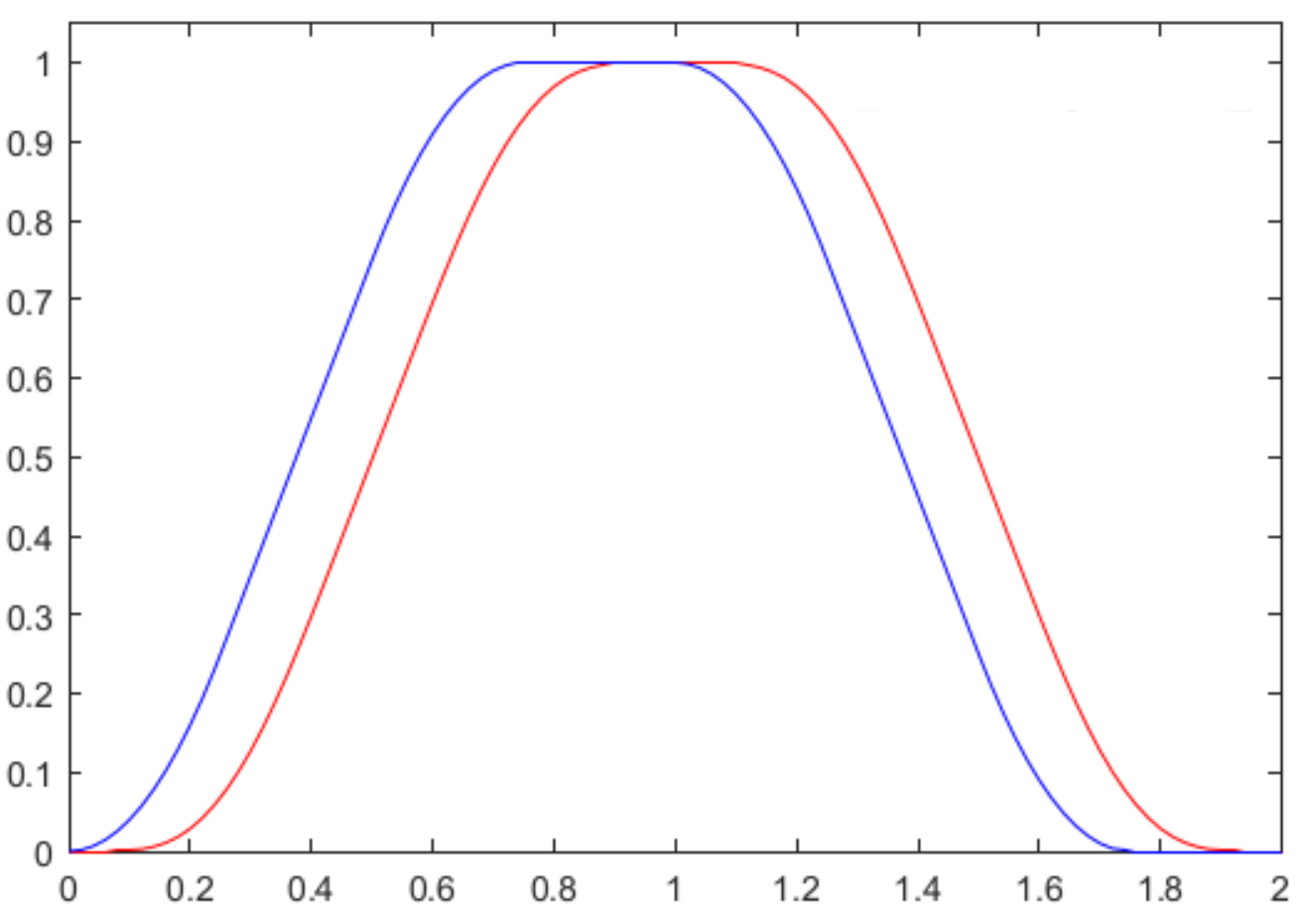}	
		
		\includegraphics[scale=0.6]{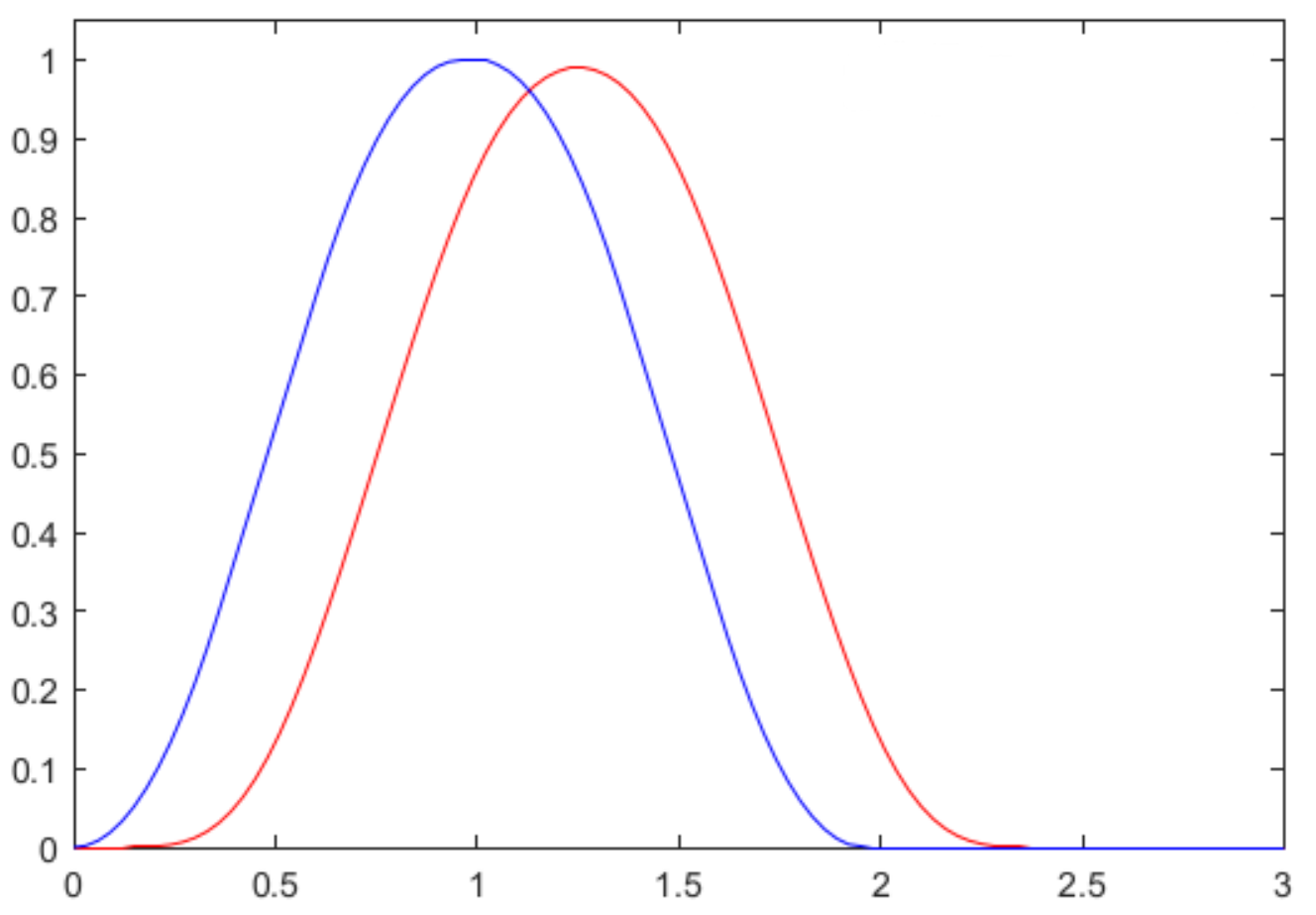}
\vspace*{-2mm}
\caption{(Unif) For $\varrho=0.4, 0.5,0.6$ with $\delta_{k}=0.02$: graphs of the estimated pdf $\mathfrak{p}_{k}$ (red)  and $\nu_{3,\varrho}$ (blue)}
\label{Unif_rho04}
\end{figure}
\begin{figure}[h]
	\centering
		\includegraphics[scale=0.6]{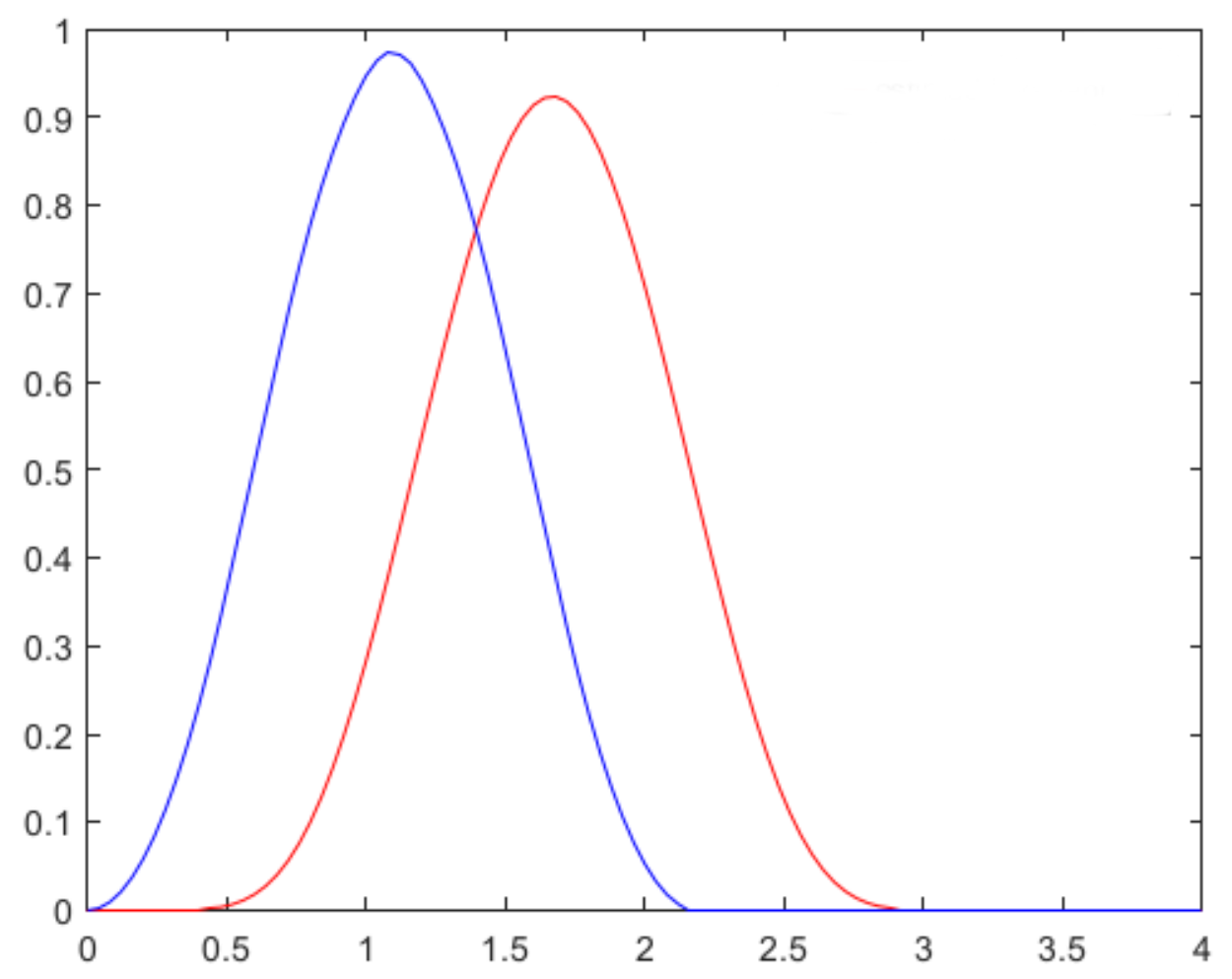}	
		
		\includegraphics[scale=0.6]{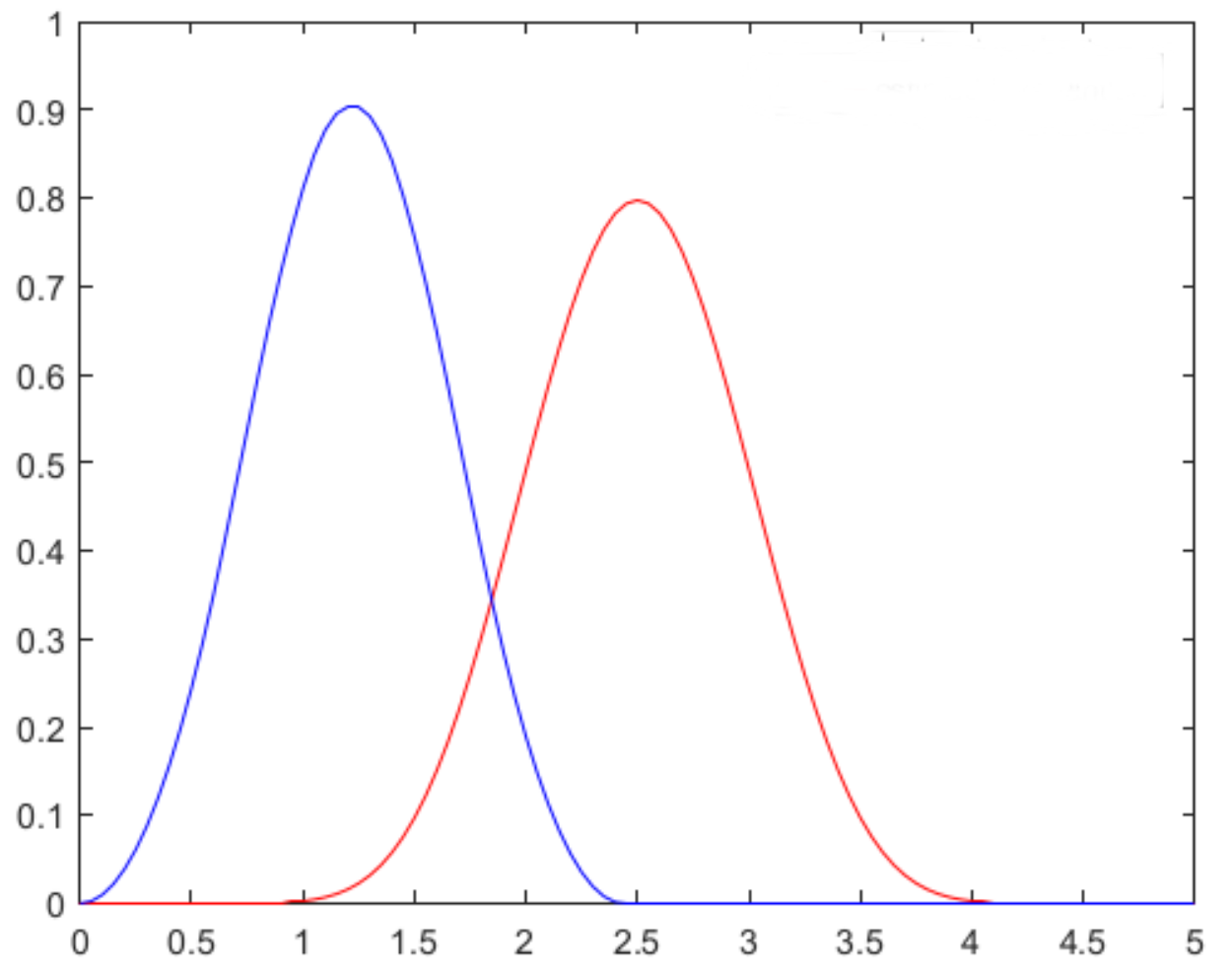}	
		
		\includegraphics[scale=0.6]{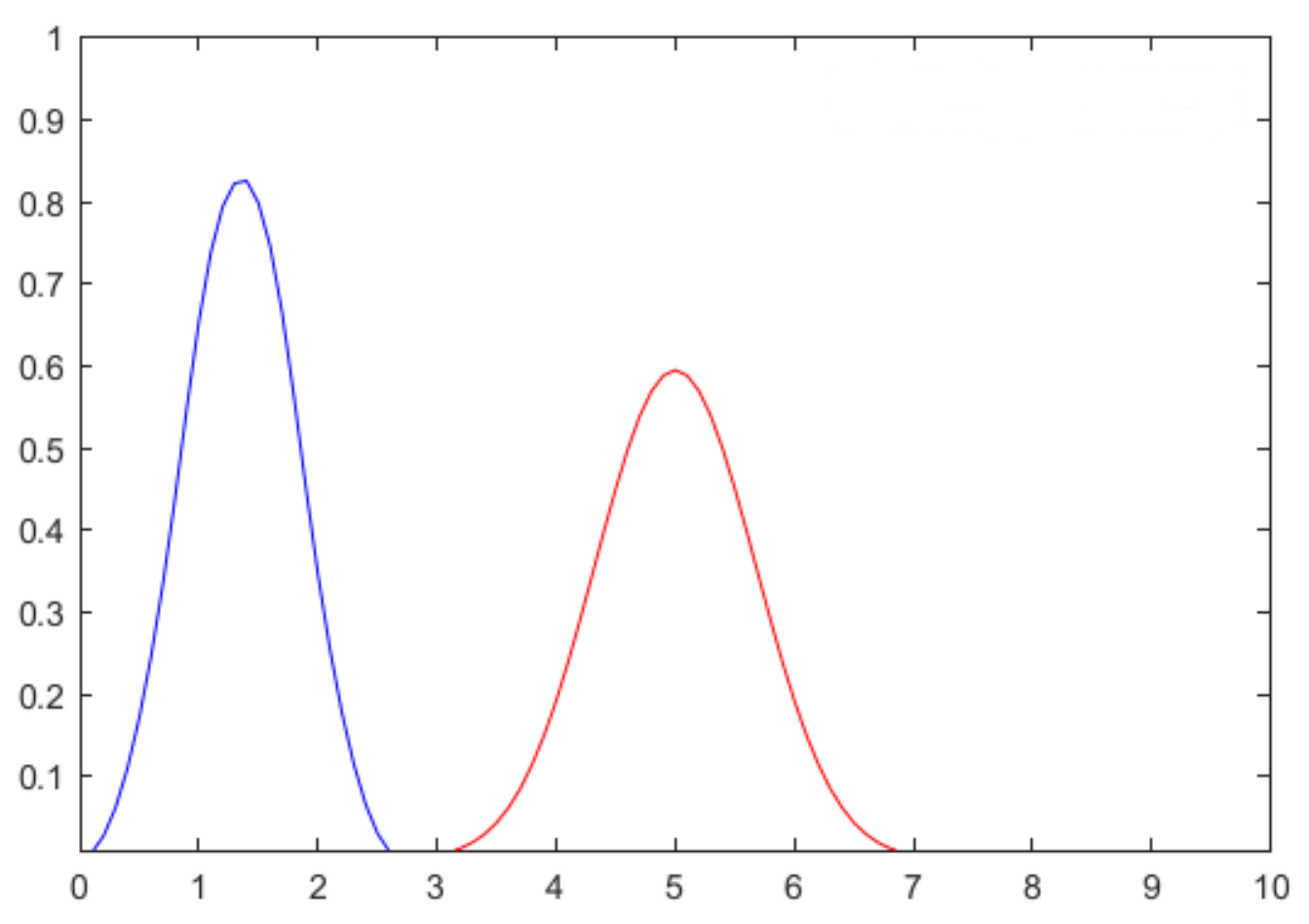}
\vspace*{-2mm}
\caption{(Unif) For $\varrho=0.7, 0.8,0.9$ with $\delta_{k}=0.02$: graphs of the estimated pdf $\mathfrak{p}_{k}$ (red)  and $\nu_{3,\varrho}$ (blue)}
\label{Unif_rho07}
\end{figure}

\subsubsection{AR(1) with ARCH(1) errors} \label{subsec-num-arch} 
In this part, we apply our generic algorithm to the autoregressive model with ARCH(1) errors and transition kernel defined in (\ref{AR-arch}). The innovation distribution is the standard Gaussian one, that is $ \vartheta_n \sim \mathcal{N}(0,1)$. The estimated invariant pdf $\mathfrak{p}_{15}$ with support $\X_{15} = [-15,15]$ and the Gaussian pdf are reported in Figure~\ref{ARCH} when $(\beta, \alpha, \lambda)=(1, 0.7, 0.2)$. 
As for AR(1) models, the invariance property of the estimated pdf $\mathfrak{p}_{15}$ is evaluated by $w_{15} := \|\big(\mathfrak{p}_{15}(x_{i,15}) - \widetilde{\mathfrak{p}}_{15}(x_{i,15})\big)_{i=0,\ldots,q_{15}}\|_{\infty}= 0.0223$. 
\begin{figure}[h]
	\centering
\includegraphics[scale=0.6]{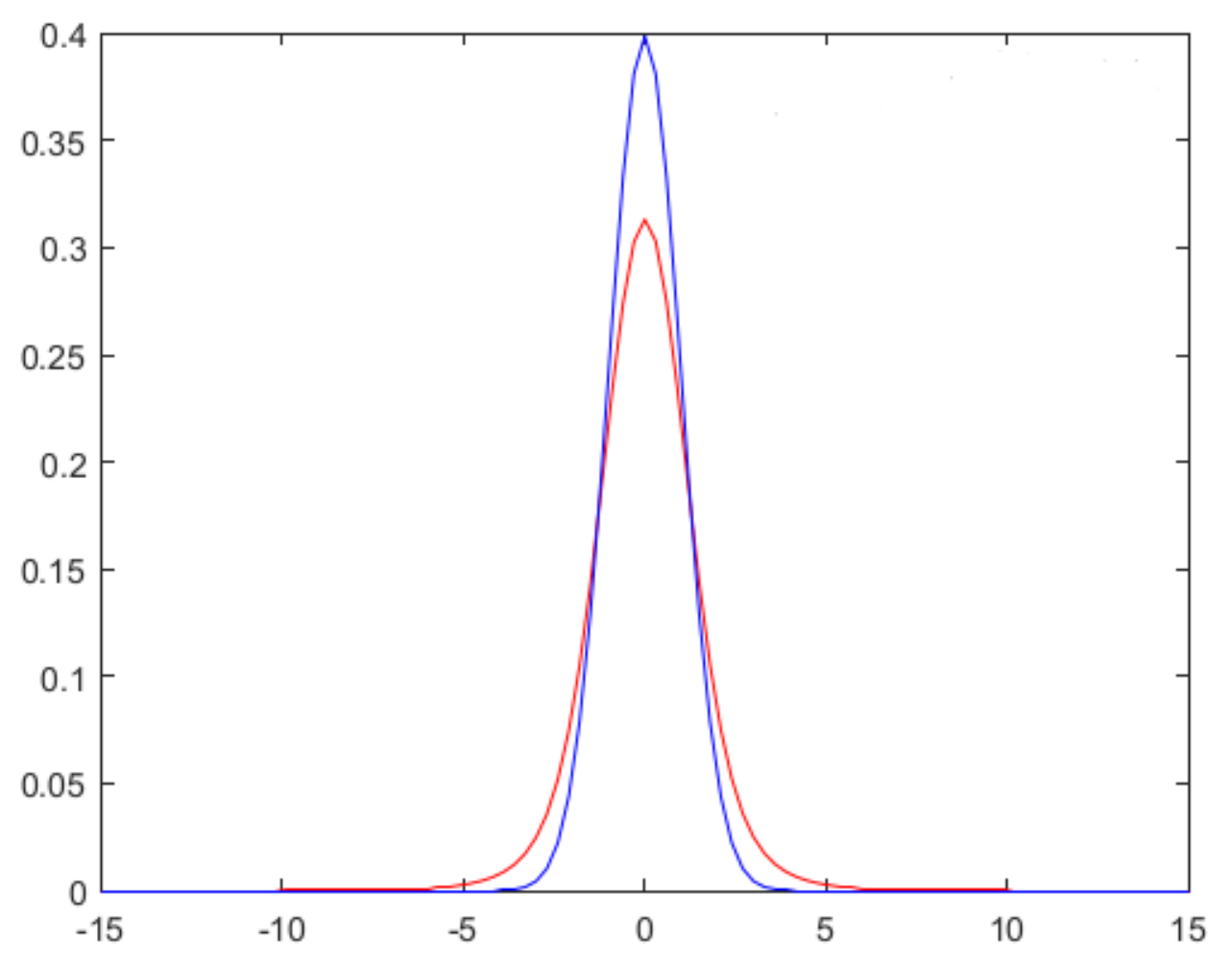}
\caption{(ARCH) For $\delta_{15}=0.02$: graphs of the estimated pdf $\mathfrak{p}_{15}$ (red)  and the innovation pdf (blue)}
\label{ARCH}
\end{figure}


\begin{thebibliography}{DDGMR00}

\bibitem[AH00]{AndHra00}
Ji{\v{r}}{\'{\i}} And{\v{e}}l and Karel Hrach.
\newblock On calculation of stationary density of autoregressive processes.
\newblock {\em Kybernetika (Prague)}, 36(3):311--319, 2000.

\bibitem[ANR07]{AndHra07}
J.~And{\v{e}}l, I.~Netuka, and P.~Ranocha.
\newblock Methods for calculating stationary distribution in linear models of
  time series.
\newblock {\em Statistics}, 41(4):279--287, 2007.

\bibitem[AR05]{AndRan05}
J.~And{\v{e}}l and P.~Ranocha.
\newblock Stationary distribution of absolute autoregression.
\newblock {\em Kybernetika (Prague)}, 41(6):735--742, 2005.

\bibitem[BK01]{BorKlu01}
M.~Borkovec and C.~Kl{\"u}ppelberg.
\newblock The tail of the stationary distribution of an autoregressive process
  with {${\rm ARCH}(1)$} errors.
\newblock {\em Ann. Appl. Probab.}, 11(4):1220--1241, 2001.

\bibitem[CT86]{ChaTon86}
K.~S. Chan and H.~Tong.
\newblock A note on certain integral equations associated with nonlinear time
  series analysis.
\newblock {\em Probab. Theory Relat. Fields}, 73(1):153--158, 1986.

\bibitem[DDGMR00]{DonGooMidRae00}
J.~A. De~Don{\'a}, G.~C. Goodwin, R.~H. Middleton, and I.~Raeburn.
\newblock Convergence of eigenvalues in state-discretization of linear
  stochastic systems.
\newblock {\em SIAM J. Matrix Anal. Appl.}, 21(4):1102--1111, 2000.

\bibitem[FHL13]{FerHerLed13}
D.~Ferr\'e, L.~Herv\'e, and J.~Ledoux.
\newblock Regular perturbation of ${V}$-geometrically ergodic {M}arkov chains.
\newblock {\em J. Appl. Probab.}, 50:184--194, 2013.

\bibitem[Hai98]{Hai98}
G.~Haiman.
\newblock Upper and lower bounds for the tail of the invariant distribution of
  some {${\rm AR}(1)$} processes.
\newblock In {\em Asymptotic methods in probability and statistics ({O}ttawa,
  {ON}, 1997)}, pages 723--730. North-Holland, Amsterdam, 1998.

\bibitem[HL14]{HerLed14}
L.~Herv\'e and J.~Ledoux.
\newblock Approximating {M}arkov chains and ${V}$-geometric ergodicity via weak
  perturbation theory.
\newblock {\em Stochastic Processes and their Applications}, 124:613--638,
  2014.

\bibitem[Kel82]{Kel82}
G.~Keller.
\newblock Stochastic stability in some chaotic dynamical systems.
\newblock {\em Monatsh. Math.}, 94(4):313--333, 1982.

\bibitem[KL99]{KelLiv99}
G.~Keller and C.~Liverani.
\newblock Stability of the spectrum for transfer operators.
\newblock {\em Annali della Scuola Normale Superiore di Pisa - Classe di
  Scienze S\'er. 4}, 28:141--152, 1999.

\bibitem[Log04]{Log04}
W.~Loges.
\newblock The stationary marginal distribution of a threshold {AR}(1) process.
\newblock {\em J. Time Ser. Anal.}, 25(1):103--125, 2004.

\bibitem[MT93]{MeyTwe93}
S.~P. Meyn and R.~L. Tweedie.
\newblock {\em Markov chains and stochastic stability}.
\newblock Springer-Verlag London Ltd., London, 1993.

\bibitem[RS18]{RudSch18}
D.~Rudolf and N.~Schweizer.
\newblock Perturbation theory for {M}arkov chains via {W}asserstein distance.
\newblock {\em Bernoulli}, 24(4A):2610--2639, 2018.

\bibitem[SS00]{ShaStu00}
T.~Shardlow and A.~M. Stuart.
\newblock A perturbation theory for ergodic {M}arkov chains and application to
  numerical approximations.
\newblock {\em SIAM J. Numer. Anal.}, 37:1120--1137, 2000.

\bibitem[{Tru}17]{Tru17}
L.~{Truquet}.
\newblock {A perturbation analysis of some Markov chains models with
  time-varying parameters}.
\newblock {\em ArXiv e-prints}, June 2017.

\end{thebibliography}
\end{document}